\documentclass[a4paper,fleqn]{cas-sc}
\usepackage[numbers]{natbib}
\usepackage{amsthm}
\usepackage[all,pdf]{xy}
\usepackage{bm}
\usepackage{extarrows}
\usepackage{algorithm}
\usepackage{algpseudocode} 

\graphicspath{{./}{./fig/}}
\usepackage{subcaption}
\DeclareMathOperator*{\argmin}{argmin}

\newtheorem{theorem}{Theorem}

\newtheorem{assumption}{Assumption}
\newtheorem{remark}{Remark}

\newtheorem{corollary}{Corollary}
\usepackage{placeins} 

\newcommand{\rd}{\mathrm{d}}

\begin{document}
\let\WriteBookmarks\relax
\def\floatpagepagefraction{1}
\def\textpagefraction{.001}

\shorttitle{Deep-Picard Iteration for Space-time Fractional Diffusion PDEs}
\shortauthors{Zeng et al.}

\title[mode=title]{Solving High-dimensional Nonlinear Space-time Fractional Diffusion PDEs via Deep-Picard Iteration}

\author[1]{Zhijun Zeng}
\author[1]{Zhitong Chen}
\cormark[1]
\author[2,3]{Ling Qin}
\author[3,4]{Yi Zhu}
\cormark[2]
\ead{yizhu@tsinghua.edu.cn}

\affiliation[1]{organization={Department of Mathematical Sciences, Tsinghua University},
            city={Beijing},
            postcode={100084},
            country={China}}
\affiliation[2]{organization={Qiuzhen College, Tsinghua University},
            city={Beijing},
            postcode={100084},
            country={China}}
\affiliation[3]{organization={Yau Mathematical Sciences Center, Tsinghua University},
            city={Beijing},
            postcode={100084},
            country={China}}
\affiliation[4]{organization={Yanqi Lake Beijing Institute of Mathematical Sciences and Applications},
            city={Beijing},
            postcode={101408},
            country={China}}

\cortext[1]{Zhijun Zeng and Zhitong Chen contribute to this work equally.}
\cortext[2]{Corresponding author}

\begin{abstract}
We propose a Deep-Picard iteration framework for high-dimensional nonlinear
space-time fractional diffusion equations.The method is based on a nonlinear fractional
Feynman--Kac fixed-point formulation, which replaces direct discretization of
the Caputo memory term and the nonlocal fractional Laplacian by Monte Carlo
simulation of the associated fractional dynamics. Each Picard update is
approximated by stochastic label generation and realized through supervised
neural-network regression, thereby avoiding residual minimization involving
fractional differential operators. The fractional trajectories are generated by
coupling a discretized $\beta$-stable subordinator with a walk-on-spheres-type
simulation of the rotationally symmetric $\alpha$-stable L\'evy process. Numerical experiments on two-dimensional and high-dimensional
test problems ddemonstrate stable Picard convergence and accurate approximation, with tests
reported up to dimension $d=100$.
\end{abstract}


\begin{keywords}
Space-time fractional diffusion equations \sep Deep Picard iteration \sep Fractional Feynman--Kac representation \sep Walk-on-spheres \sep High-dimensional PDEs
\end{keywords}

\maketitle

\section{Introduction}

We study the numerical solution of high-dimensional nonlinear space-time fractional partial differential equations of the form
\begin{equation}\label{eq:intro-pde}
\begin{cases}
\partial_t^\beta u(t,x) + (-\Delta)^{\alpha/2}u(t,x) = f\bigl(t,u(t,x),x\bigr), & (t,x)\in (0,T]\times\Omega,\\
u(t,x)=g(t,x), & (t,x)\in [0,T]\times \Omega^c,\\
u(0,x)=u_0(x), & x\in \Omega,
\end{cases}
\end{equation}
where $\Omega\subset\mathbb{R}^d$ is a bounded domain, 
$\Omega^c:=\mathbb{R}^d\setminus\Omega$ denotes its exterior, 
$\partial_t^\beta$ is the Caputo fractional derivative defined in 
\eqref{eq:caputo-def}, $(-\Delta)^{\alpha/2}$ is the integral fractional 
Laplacian defined in \eqref{eq:fraclap-def}, and the fractional orders satisfy 
$\beta\in(0,1)$ and $\alpha\in(0,2)$. Such equations arise in practical models where classical diffusion fails, including contaminant transport in heterogeneous porous media, non-diffusive transport in plasma turbulence, and biological cell migration in complex spatial environments. In these settings, the Caputo derivative captures trapping and long-memory effects, while the fractional Laplacian represents nonlocal transport driven by long-range jumps or L\'evy-flight dynamics
\cite{BensonWheatcraftMeerschaert2000FractionalADE,
delCastilloNegreteCarrerasLynch2004PlasmaFractionalDiffusion,
CusimanoBurrageBurrage2013CellMigrationFractional}. From a probabilistic perspective, Caputo time-fractional evolution is naturally associated with inverse stable time changes, while, on the whole space $\mathbb{R}^d$, the operator $-(-\Delta)^{\alpha/2}$ is the infinitesimal generator of a rotationally symmetric $\alpha$-stable L\'evy process \cite{BaeumerMeerschaert2001StochasticCauchy,MeerschaertNaneVellaisamy2009BoundedCauchy,Kwasnicki2017TenDefs}. The integral fractional Laplacian is spatially nonlocal, and in the bounded-domain setting its Dirichlet formulation naturally involves exterior data on $\Omega^c$, which distinguishes the problem from classical local parabolic equations \cite{LischkePangEtal2020WhatFracLaplacian,DEliaGunzburger2013BoundedFracLap}.

The numerical solution of \eqref{eq:intro-pde} is challenging for three closely related reasons. First, the Caputo derivative introduces history dependence, and its direct discretization typically entails substantial storage and computational cost. Second, the integral fractional Laplacian couples each spatial point to the entire exterior domain and typically produces dense algebraic structures after discretization. Third, the nonlinear reaction term precludes an explicit linear Feynman--Kac representation and instead leads to a nonlinear operator equation of fixed-point type. As a result, classical discretization methods face not only the usual curse of dimensionality, but also additional difficulties caused by time-fractional memory and spatial nonlocality. These difficulties are further amplified in high dimensions, where accurate numerical approximation itself becomes highly nontrivial.

A large body of work has been devoted to deterministic numerical methods for fractional PDEs. For the time-fractional derivative, representative approaches include $\mathrm{L1}$-type schemes, graded meshes for resolving the initial singularity, convolution quadrature, and fast memory-saving variants based on exponential-sum approximations \cite{Lubich1986CQ,SchadleLopezFernandezLubich2006FastObliviousCQ,StynesORiordanGracia2017L1Graded,JinLazarovZhou2019Overview,jin2023numerical,jin2021fractional}. For the integral fractional Laplacian, widely used discretizations include finite difference--quadrature methods, finite element methods, and fast solvers for the resulting dense or nonlocal linear systems \cite{HuangOberman2014FDQuadratureFracLaplacian,AcostaBorthagaray2017FEMFracLap,MindenYing2020SimpleSolverFracLap,LischkePangEtal2020WhatFracLaplacian}. For space-time fractional diffusion problems, one may combine these temporal and spatial discretization strategies, and a variety of coupled schemes have also been developed; see, for example, \cite{YangTurnerLiuIlic2011TimeSpace}. These methods are effective in low and moderate dimensions, but their computational cost typically grows rapidly with dimension, and the simultaneous presence of time-fractional memory and spatial nonlocality often becomes the dominant bottleneck.

An alternative route is provided by probabilistic representations and Monte Carlo methods. For linear time-fractional and space-time fractional diffusion equations, stochastic representations in terms of stable subordinators and stable L\'evy processes are well established \cite{BaeumerMeerschaert2001StochasticCauchy,MeerschaertBensonSchefflerBaeumer2002StochasticSpaceTime,MeerschaertNaneVellaisamy2009BoundedCauchy,GorenfloMainardiEtAl2009Simulation}. For the integral fractional Laplacian on bounded domains, walk-on-spheres and related Monte Carlo methods are especially attractive because they are mesh-free and naturally compatible with exterior-value formulations \cite{KyprianouOsojnikShardlow2017UnbiasedWoS,Shardlow2019WalkOutsideSpheres,ShengSuXu2023EfficientMCIFL,LischkePangEtal2020WhatFracLaplacian}. Very recently, stochastic algorithms have also been proposed for bounded-domain space-time fractional diffusion problems, further illustrating the potential of this perspective in high dimensions \cite{CuiShengSuZhou2025StochasticST}. Nevertheless, most existing stochastic algorithms for space--time fractional diffusion are developed for linear equations, and do not
directly address the nonlinear fixed-point structure induced by reaction terms
of the form $f(t,u,x)$.

In parallel, deep-learning-based methods have broadened the range of computational approaches available for high-dimensional PDEs. Representative examples include the Deep BSDE method \cite{EHanJentzen2017DeepLearningBSDE,HanJentzenE2018DeepBSDE}, residual-based approaches such as the Deep Galerkin Method and physics-informed neural networks \cite{SirignanoSpiliopoulos2018DGM,RaissiPerdikarisKarniadakis2019PINN}, and fractional extensions such as fPINNs \cite{PangLuKarniadakis2019fPINNs}. These methods are highly flexible, but for fractional problems they typically require repeated evaluation of nonlocal operators and history-dependent residual terms during training, which can be computationally expensive and may also complicate optimization. A complementary line of work is provided by Picard-type methods and multilevel Picard schemes, which reduce high-dimensional semilinear PDEs to a sequence of expectation-based subproblems \cite{Giles2008MLMC,EHutzenthalerJentzenKruse2021MLPHeat}. More recently, Deep Picard Iteration (DPI) has demonstrated that Picard iteration and neural-network regression can be combined to replace PDE-residual objectives by supervised regression problems in high-dimensional nonlinear PDEs \cite{HanHuLongZhao2024DPI}. However, existing Picard-based methods are developed mainly for Markovian semilinear PDEs, whereas inverse-stable time changes and jump-driven spatial dynamics make the present fractional setting substantially harder to treat numerically.

In this work, we develop a deep Picard iteration framework for nonlinear space-time fractional diffusion equations in high dimensions. The method is based on the nonlinear fractional Feynman--Kac fixed-point formulation, which recasts the original problem as a stochastic fixed-point equation. This formulation replaces the direct discretization of the Caputo time-fractional derivative and the nonlocal fractional Laplacian by a sequence of Picard updates. For the numerical realization, we combine Monte Carlo simulation of the underlying fractional dynamics with neural-network regression. More specifically, the temporal component is treated through a discretization of the $\beta$-stable subordinator, while the spatial component is approximated by a walk-on-spheres-type simulation of the rotationally symmetric $\alpha$-stable L\'evy process. The resulting stochastic trajectories are then used to construct regression targets for successive neural approximations of the Picard iterates. In this way, the nonlinear fractional PDE is converted into a sequence of supervised learning problems while preserving the underlying probabilistic structure of the model.

Compared with classical discretization methods, the proposed framework has several computational advantages:
\begin{itemize}
    \item \textbf{Mesh-free treatment of fractional effects.}
    The method avoids direct discretization of both the Caputo time-fractional
    derivative and the nonlocal fractional Laplacian. Instead, the temporal
    memory and spatial nonlocality are represented through stochastic simulation
    of the underlying fractional dynamics, namely the $\beta$-stable
    subordinator and the rotationally symmetric $\alpha$-stable L\'evy process.
    This Monte Carlo formulation preserves the mesh-free character of the
    stochastic representation and avoids the dense linear systems and storage
    overhead that often arise in deterministic discretizations of fractional
    operators.
    
    \item \textbf{Natural treatment of nonlinearities.}
    The nonlinear reaction term is incorporated through the nonlinear
    Feynman--Kac fixed-point formulation rather than through a residual
    involving fractional differential operators. Successive Picard updates of
    this fixed-point equation are then approximated by neural-network
    regression using Monte Carlo labels. In this way, the nonlinear fractional
    PDE is converted into a sequence of supervised learning problems, making
    the framework naturally suited to semilinear fractional equations and
    potentially extensible to more general nonlinear models.

    \item \textbf{Scalability in high dimensions.}
    The method is well suited to high-dimensional problems because the Monte Carlo approximation of each Picard update is trajectory-based, naturally parallelizable, and not tied to tensor-product spatial grids. Consequently, the method avoids tensor-product spatial grids and the dense
algebraic structures typical of grid-based nonlocal discretizations. This feature is supported by the numerical results reported in Section~\ref{sec:numexp}, where stable and accurate performance is observed up to dimension $d=100$.
\end{itemize}

\section{Preliminaries}\label{sec:prelim}

\subsection{Space-time fractional operators and problem setting}
\label{subsec:frac-ops}

Throughout the paper, we fix the temporal order $\beta\in(0,1)$ and the spatial order 
$\alpha\in(0,2)$, and consider the nonlinear space-time fractional Dirichlet 
problem \eqref{eq:intro-pde} on a bounded domain $\Omega\subset\mathbb{R}^d$ with 
final time $T>0$.

The \emph{Caputo fractional derivative} of order $\beta$ is defined, for sufficiently regular $u(\cdot,x)$, by
\begin{equation}\label{eq:caputo-def}
\partial_t^\beta u(t,x)
=\frac{1}{\Gamma(1-\beta)}\int_0^{t}(t-s)^{-\beta}\,\partial_s u(s,x)\,ds,
\qquad t>0,
\end{equation}
where $\Gamma(\cdot)$ is the Euler Gamma function. The \emph{integral fractional Laplacian} of order $\alpha/2$ is defined, for $u\in\mathcal{S}(\mathbb{R}^d)$, as the principal-value singular integral
\begin{equation}\label{eq:fraclap-def}
(-\Delta)^{\alpha/2} u(x)
= c_{d,\alpha}\,\mathrm{P.V.}\!\int_{\mathbb{R}^d}\frac{u(x)-u(y)}{|x-y|^{d+\alpha}}\,dy,
\qquad
c_{d,\alpha}=\frac{2^{\alpha}\,\Gamma\!\left(\tfrac{d+\alpha}{2}\right)}{\pi^{d/2}\,\bigl|\Gamma(-\alpha/2)\bigr|},
\end{equation}
and extended to more general functions by density. Equivalently, it is the Fourier multiplier
\begin{equation}\label{eq:fraclap-fourier}
\mathcal{F}\bigl[(-\Delta)^{\alpha/2}u\bigr](\xi)
=
|\xi|^{\alpha}\,\widehat u(\xi),
\qquad \xi\in\mathbb{R}^d,
\end{equation}
where $\widehat u=\mathcal{F}u$ denotes the Fourier transform of $u$. It is well known that the operator $-(-\Delta)^{\alpha/2}$ is the
infinitesimal generator of the rotationally symmetric $\alpha$-stable
L\'evy process on $\mathbb{R}^d$, while Caputo time-fractional evolution is
naturally associated with inverse stable time changes generated by
$\beta$-stable subordinators; see, e.g.,
\cite{MeerschaertBensonSchefflerBaeumer2002StochasticSpaceTime,
BaeumerMeerschaert2001StochasticCauchy}.

Under standard assumptions on $(f,g,u_0)$, well-posedness and regularity results for time-fractional and space-time fractional diffusion problems have been established in several settings; see, for example, \cite{Zacher2009WeakSolutions,KSVZ2016DecayEstimates,RosOtonSerra2014Dirichlet,LeonenkoMeerschaertSikorskii2013Fractional,ChenKimSong2012HeatKernels}. These results provide the analytical foundation for the nonlinear probabilistic representation used below.

\subsection{Feynman--Kac-type stochastic fixed-point formulation}
\label{subsec:nonlinear-fk}
We next introduce the stochastic processes that underlie the probabilistic
formulation of \eqref{eq:intro-pde}. The spatial component is tied to the integral fractional Laplacian through the
generator of a rotationally symmetric $\alpha$-stable L\'evy process. Let
$\{X_s^x\}_{s\ge0}$ be such a process in $\mathbb{R}^d$ started from $x\in\Omega$.
Its characteristic function satisfies
\begin{equation}\label{eq:alpha-stable-char}
\mathbb{E}\!\left[
\exp\!\left(i\xi\cdot (X_s^x-x)\right)
\right]
=
\exp\!\left(-s|\xi|^\alpha\right),
\qquad
\xi\in\mathbb{R}^d,\quad s\ge0 .
\end{equation}
Consequently, the infinitesimal generator of $X_s^x$ is
$-(-\Delta)^{\alpha/2}$. For $\alpha=2$, this process reduces, up to a scaling
convention, to Brownian motion, whereas for $\alpha\in(0,2)$ it is a pure-jump
L\'evy process with heavy-tailed increments.

The temporal component is constructed from a one-sided $\beta$-stable
subordinator. Let $\{Y_s^\beta\}_{s\ge0}$ be a $\beta$-stable subordinator,
independent of $X_s^x$, characterized by the Laplace transform
\begin{equation}\label{eq:stable-laplace}
\mathbb{E}\!\left[e^{-\lambda Y_s^\beta}\right]
=
e^{-s\lambda^\beta},
\qquad
\lambda>0,\quad s\ge0 .
\end{equation}
The process $Y_s^\beta$ is nondecreasing and has stationary independent
increments. In time-fractional diffusion, the Caputo derivative is naturally
associated with the first-passage mechanism of such stable subordinators.
For a fixed physical time $t\in[0,T]$, we therefore run $Y_s^\beta$ forward in
the operational time $s$ and stop it when it crosses the level $t$. Equivalently,
we introduce the backward stochastic clock
\begin{equation}\label{eq:physical-clock}
t_s := t-Y_s^\beta .
\end{equation}
As $s$ increases, the accumulated subordinator moves the physical clock
backward from $t$ toward the initial time. The corresponding time-crossing
event is
\begin{equation}\label{eq:exit-time-fractional}
\tau_t
:=
\inf\{s>0:\,t_s\le0\}
=
\inf\{s>0:\,Y_s^\beta\ge t\}.
\end{equation}

The bounded-domain problem requires stopping the coupled trajectory when either
the spatial motion leaves the domain or the stochastic clock reaches the initial
time. We define the spatial exit time
\begin{equation}\label{eq:exit-space}
\tau_\Omega(x)
:=
\inf\{s>0:\,X_s^x\notin\Omega\},
\end{equation}
and combine it with the time-crossing event \(\tau_t\) through
\begin{equation}\label{eq:tau-nonlinear}
\tau(t,x):=\tau_t\wedge\tau_\Omega(x).
\end{equation}
If $\tau_t<\tau_\Omega(x)$, the path reaches the initial-time surface before leaving
the spatial domain, and the initial datum is evaluated. If
$\tau_\Omega(x)\le\tau_t$, the spatial process exits $\Omega$ at a positive physical
time, and the exterior Dirichlet datum is evaluated. Accordingly, we define the
terminal payoff
\begin{equation}\label{eq:terminal-payoff-prelim}
\mathcal{G}_{t,x}
=
u_0\!\bigl(X_{\tau_t}^x\bigr)\mathbf{1}_{\{\tau_t<\tau_\Omega(x)\}}
+
g\!\bigl(t_{\tau_\Omega(x)},X_{\tau_\Omega(x)}^x\bigr)
\mathbf{1}_{\{\tau_\Omega(x)\le\tau_t\}},
\qquad
t_{\tau_\Omega(x)}=t-Y_{\tau_\Omega(x)}^\beta .
\end{equation}

For linear fractional diffusion equations, the above stopped process gives a direct
Feynman--Kac representation. In the present nonlinear problem, the reaction term
depends on the unknown solution itself. Hence the corresponding formula is implicit:
under suitable regularity and integrability assumptions, the solution of
\eqref{eq:intro-pde} satisfies the Feynman--Kac-type stochastic fixed-point equation
\begin{equation}\label{eq:nonlinear-fk}
u(t,x)
=
\mathbb{E}_x\!\left[
\mathcal{G}_{t,x}
+
\int_0^{\tau(t,x)}
f\!\bigl(t_s,u(t_s,X_s^x),X_s^x\bigr)\,\rd s
\right],
\qquad (t,x)\in[0,T]\times\Omega .
\end{equation}
Equivalently, defining the nonlinear expectation operator
\begin{equation}\label{eq:T-operator-prelim}
(\mathcal{T}v)(t,x)
:=
\mathbb{E}_x\!\left[
\mathcal{G}_{t,x}
+
\int_0^{\tau(t,x)}
f\!\bigl(t_s,v(t_s,X_s^x),X_s^x\bigr)\,\rd s
\right],
\end{equation}
the PDE solution is characterized by the fixed-point relation
\begin{equation}\label{eq:fixed-point-prelim}
u=\mathcal{T}u .
\end{equation}

This formulation should therefore be viewed as a stochastic fixed-point equation
rather than as an explicit evaluation formula. Successive Picard updates of
$\mathcal{T}$, with Monte Carlo regression supplying labels along the stopped,
time-changed, jump-driven trajectory $(t_s,X_s^x)$, provide a natural algorithmic
strategy, which is developed in Section~\ref{sec:method}. Before turning to the
algorithm, the next subsection establishes the well-posedness of
\eqref{eq:nonlinear-fk} through a contraction estimate for $\mathcal{T}$.

\subsection{A weighted contraction result}
\label{subsec:weighted-contraction}

We now establish that the operator $\mathcal{T}$ in \eqref{eq:T-operator-prelim},
together with the terminal payoff $\mathcal{G}_{t,x}$ in
\eqref{eq:terminal-payoff-prelim}, is a strict contraction on a suitably weighted
Banach space. Under an exponentially weighted sup-norm, the contraction constant
is independent of the size of the stopping time $\tau(t,x)$, which removes the
restrictive smallness condition based on
$\sup_{(t,x)}\mathbb{E}[\tau(t,x)]$ that typically appears in the unweighted
setting.

For $\lambda>0$, define the weighted space
\[
\mathcal{X}_\lambda
:=
\Bigl\{
v:[0,T]\times\overline{\Omega}\to\mathbb{R}\ \text{measurable}:\ 
\|v\|_\lambda<\infty
\Bigr\},
\]
equipped with the norm
\begin{equation}\label{eq:weighted-norm-final}
\|v\|_\lambda
:=
\sup_{(t,x)\in[0,T]\times\overline{\Omega}}
e^{-\lambda t}|v(t,x)|.
\end{equation}
Since $T<\infty$, the weighted norm $\|\cdot\|_\lambda$ is equivalent to the usual sup-norm on $[0,T]\times\overline{\Omega}$, and hence $(\mathcal{X}_\lambda,\|\cdot\|_\lambda)$ is a Banach space.

We impose the following assumptions.

\begin{assumption}\label{ass:weighted-contraction}
The following conditions hold.
\begin{itemize}
    \item[(A1)] The nonlinearity $f:[0,T]\times\mathbb{R}\times\Omega\to\mathbb{R}$ is measurable and globally Lipschitz continuous in its second variable, uniformly in $(t,x)$, i.e.,
    \begin{equation}\label{eq:f-global-lip-final}
    |f(t,u,x)-f(t,v,x)|
    \le L_f |u-v|,
    \qquad
    \forall (t,x)\in[0,T]\times\Omega,\ \forall u,v\in\mathbb{R},
    \end{equation}
    for some constant $L_f>0$.
    
    \item[(A2)] The zero-level section of $f$ is bounded:
    \begin{equation}\label{eq:f-zero-bounded-final}
    M_f:=\sup_{(t,x)\in[0,T]\times\Omega}|f(t,0,x)|<\infty.
    \end{equation}
    
    \item[(A3)] The initial datum and boundary datum are bounded:
    \begin{equation}\label{eq:data-bounded-final}
    M_0:=\sup_{x\in\Omega}|u_0(x)|<\infty,
    \qquad
    M_g:=\sup_{(t,x)\in[0,T]\times\Omega^c}|g(t,x)|<\infty.
    \end{equation}
\end{itemize}
\end{assumption}
\begin{remark}
The global Lipschitz condition in Assumption~\ref{ass:weighted-contraction}
is a sufficient condition for the contraction argument. Some polynomial
nonlinearities in Section~\ref{sec:numexp} do not satisfy this condition
globally, and are included as numerical tests beyond the theoretical setting.
\end{remark}

We now state the main result of this subsection.

\begin{theorem}[Weighted contraction of the nonlinear Feynman--Kac operator]
\label{thm:weighted-contraction-final}
Suppose that Assumption~\ref{ass:weighted-contraction} holds. Then, for every $\lambda>0$, the operator $\mathcal{T}$ maps $\mathcal{X}_\lambda$ into itself. Moreover, for all $v,w\in\mathcal{X}_\lambda$,
\begin{equation}\label{eq:weighted-contraction-est-final}
\|\mathcal{T}v-\mathcal{T}w\|_\lambda
\le
\frac{L_f}{\lambda^\beta}\,\|v-w\|_\lambda .
\end{equation}
In particular, if
\begin{equation}\label{eq:lambda-condition-final}
\lambda^\beta>L_f,
\end{equation}
then $\mathcal{T}$ is a strict contraction on $\mathcal{X}_\lambda$.
\end{theorem}

\begin{proof}
The proof is divided into two steps.

\medskip
\noindent
\textbf{Step 1. Weighted Lipschitz estimate.}
Let $v,w\in\mathcal{X}_\lambda$ be arbitrary. Fix $(t,x)\in[0,T]\times\Omega$. Since the terminal payoff $\mathcal{G}_{t,x}$ is independent of the function argument of $\mathcal{T}$, we have
\begin{align}
(\mathcal{T}v)(t,x)-(\mathcal{T}w)(t,x)
&=
\mathbb{E}_x\!\left[
\int_0^{\tau(t,x)}
\Bigl(
f\!\bigl(t_s,v(t_s,X_s^x),X_s^x\bigr)
-
f\!\bigl(t_s,w(t_s,X_s^x),X_s^x\bigr)
\Bigr)\,\mathrm{d}s
\right].
\label{eq:weighted-proof-diff-final}
\end{align}
Taking absolute values and using \eqref{eq:f-global-lip-final}, we obtain
\begin{align}
|(\mathcal{T}v)(t,x)-(\mathcal{T}w)(t,x)|
&\le
L_f\,
\mathbb{E}_x\!\left[
\int_0^{\tau(t,x)}
|v(t_s,X_s^x)-w(t_s,X_s^x)|
\,\mathrm{d}s
\right].
\label{eq:weighted-proof-lip-final}
\end{align}
For every $0\le s<\tau(t,x)$, by the definition of the stopping time $\tau(t,x)=\tau_t(x)\wedge \tau_\Omega(x)$, we have
\[
t_s>0,
\qquad
X_s^x\in\Omega.
\]
Hence $(t_s,X_s^x)\in[0,T]\times\overline{\Omega}$, and by the definition of the weighted norm,
\begin{equation}\label{eq:weighted-pointwise-final}
|v(t_s,X_s^x)-w(t_s,X_s^x)|
\le
e^{\lambda t_s}\|v-w\|_\lambda .
\end{equation}
Substituting \eqref{eq:weighted-pointwise-final} into \eqref{eq:weighted-proof-lip-final}, and multiplying both sides by $e^{-\lambda t}$, yield
\begin{align}
e^{-\lambda t}|(\mathcal{T}v)(t,x)-(\mathcal{T}w)(t,x)|
&\le
L_f\|v-w\|_\lambda\,
\mathbb{E}_x\!\left[
\int_0^{\tau(t,x)}
e^{-\lambda t}e^{\lambda t_s}\,\mathrm{d}s
\right]
\nonumber\\
&=
L_f\|v-w\|_\lambda\,
\mathbb{E}_x\!\left[
\int_0^{\tau(t,x)}
e^{-\lambda(t-t_s)}\,\mathrm{d}s
\right].
\label{eq:weighted-proof-after-weight-final}
\end{align}
Since $t_s=t-Y_s^\beta$, we have
\(
t-t_s=Y_s^\beta,
\)
and therefore
\begin{equation}\label{eq:weighted-proof-before-tonelli-final}
e^{-\lambda t}|(\mathcal{T}v)(t,x)-(\mathcal{T}w)(t,x)|
\le
L_f\|v-w\|_\lambda\,
\mathbb{E}_x\!\left[
\int_0^{\tau(t,x)}
e^{-\lambda Y_s^\beta}\,\mathrm{d}s
\right].
\end{equation}
Since the integrand is nonnegative, Tonelli's theorem gives
\begin{align}
\mathbb{E}_x\!\left[
\int_0^{\tau(t,x)} e^{-\lambda Y_s^\beta}\,\mathrm{d}s
\right]
&=
\mathbb{E}_x\!\left[
\int_0^\infty \mathbf{1}_{\{s<\tau(t,x)\}} e^{-\lambda Y_s^\beta}\,\mathrm{d}s
\right]
\nonumber\\
&=
\int_0^\infty
\mathbb{E}_x\!\left[
\mathbf{1}_{\{s<\tau(t,x)\}} e^{-\lambda Y_s^\beta}
\right]
\,\mathrm{d}s
\nonumber\\
&\le
\int_0^\infty
\mathbb{E}\!\left[e^{-\lambda Y_s^\beta}\right]
\,\mathrm{d}s
\nonumber\\
&=
\int_0^\infty e^{-s\lambda^\beta}\,\mathrm{d}s
=
\frac{1}{\lambda^\beta},
\label{eq:weighted-proof-core-est-final}
\end{align}
where in the last step we used \eqref{eq:stable-laplace}. Combining \eqref{eq:weighted-proof-before-tonelli-final} and \eqref{eq:weighted-proof-core-est-final}, we obtain
\[
e^{-\lambda t}|(\mathcal{T}v)(t,x)-(\mathcal{T}w)(t,x)|
\le
\frac{L_f}{\lambda^\beta}\,\|v-w\|_\lambda.
\]
Taking the supremum over $(t,x)\in[0,T]\times\Omega$ gives
\[
\|\mathcal{T}v-\mathcal{T}w\|_\lambda
\le
\frac{L_f}{\lambda^\beta}\,\|v-w\|_\lambda,
\]
which proves \eqref{eq:weighted-contraction-est-final}.

\medskip
\noindent
\textbf{Step 2. The mapping property $\mathcal{T}(\mathcal{X}_\lambda)\subset\mathcal{X}_\lambda$.}
Let $v\in\mathcal{X}_\lambda$ and fix $(t,x)\in[0,T]\times\Omega$. By \eqref{eq:T-operator-prelim},
\begin{align}
|(\mathcal{T}v)(t,x)|
&\le
\mathbb{E}_x|\mathcal{G}_{t,x}|
+
\mathbb{E}_x\!\left[
\int_0^{\tau(t,x)}
\left|f\!\bigl(t_s,v(t_s,X_s^x),X_s^x\bigr)\right|
\,\mathrm{d}s
\right].
\label{eq:mapping-step1-final}
\end{align}
By \eqref{eq:f-global-lip-final} and \eqref{eq:f-zero-bounded-final},
\[
|f(t_s,v(t_s,X_s^x),X_s^x)|
\le
L_f|v(t_s,X_s^x)|+M_f.
\]
Hence
\begin{align}
|(\mathcal{T}v)(t,x)|
&\le
\mathbb{E}_x|\mathcal{G}_{t,x}|
+
L_f\,
\mathbb{E}_x\!\left[
\int_0^{\tau(t,x)}
|v(t_s,X_s^x)|\,\mathrm{d}s
\right]
+
M_f\,\mathbb{E}_x[\tau(t,x)].
\label{eq:mapping-step2-rough-final}
\end{align}
Instead of estimating the last term by $\mathbb{E}[\tau(t,x)]$, we exploit the same exponential weight as in Step~1. Multiplying \eqref{eq:mapping-step1-final} by $e^{-\lambda t}$, and using \eqref{eq:weighted-norm-final}, we obtain
\begin{align}
e^{-\lambda t}|(\mathcal{T}v)(t,x)|
&\le
e^{-\lambda t}\mathbb{E}_x|\mathcal{G}_{t,x}|
+
L_f\|v\|_\lambda
\mathbb{E}_x\!\left[
\int_0^{\tau(t,x)}
e^{-\lambda(t-t_s)}\,\mathrm{d}s
\right]
\nonumber\\
&\qquad
+
M_f\,
\mathbb{E}_x\!\left[
\int_0^{\tau(t,x)} e^{-\lambda t}\,\mathrm{d}s
\right].
\label{eq:mapping-step3-final}
\end{align}
We next estimate the three terms on the right-hand side.

First, by \eqref{eq:data-bounded-final} and the definition \eqref{eq:terminal-payoff-prelim},
\[
|\mathcal{G}_{t,x}|
\le
M_0\,\mathbf{1}_{\{\tau_t<\tau_\Omega(x)\}}
+
M_g\,\mathbf{1}_{\{\tau_\Omega(x)\le \tau_t\}}.
\]

Therefore,
\begin{equation}\label{eq:terminal-est-final}
e^{-\lambda t}\mathbb{E}_x|\mathcal{G}_{t,x}|
\le
\max\{M_0,M_g\}.
\end{equation}

Second, the same estimate as in \eqref{eq:weighted-proof-core-est-final} yields
\begin{equation}\label{eq:mapping-second-est-final}
L_f\|v\|_\lambda
\mathbb{E}_x\!\left[
\int_0^{\tau(t,x)}
e^{-\lambda(t-t_s)}\,\mathrm{d}s
\right]
\le
\frac{L_f}{\lambda^\beta}\|v\|_\lambda.
\end{equation}

Third, for $0\le s<\tau(t,x)$ we have $t_s\ge 0$, and hence
\[
e^{-\lambda t}
=
e^{-\lambda t_s}e^{-\lambda(t-t_s)}
\le
e^{-\lambda(t-t_s)}
=
e^{-\lambda Y_s^\beta}.
\]
Thus,
\begin{align}
M_f\,
\mathbb{E}_x\!\left[
\int_0^{\tau(t,x)} e^{-\lambda t}\,\mathrm{d}s
\right]
&\le
M_f\,
\mathbb{E}_x\!\left[
\int_0^{\tau(t,x)} e^{-\lambda Y_s^\beta}\,\mathrm{d}s
\right]
\le
\frac{M_f}{\lambda^\beta}.
\label{eq:mapping-third-est-final}
\end{align}
Combining \eqref{eq:mapping-step3-final}--\eqref{eq:mapping-third-est-final}, we arrive at
\[
e^{-\lambda t}|(\mathcal{T}v)(t,x)|
\le
\max\{M_0,M_g\}
+
\frac{L_f}{\lambda^\beta}\|v\|_\lambda
+
\frac{M_f}{\lambda^\beta}.
\]
Taking the supremum over $(t,x)\in[0,T]\times\Omega$, we obtain
\begin{equation}\label{eq:mapping-bound-final}
\|\mathcal{T}v\|_\lambda
\le
\max\{M_0,M_g\}
+
\frac{L_f}{\lambda^\beta}\|v\|_\lambda
+
\frac{M_f}{\lambda^\beta}
<\infty.
\end{equation}
Hence $\mathcal{T}v\in\mathcal{X}_\lambda$, i.e., $\mathcal{T}$ maps $\mathcal{X}_\lambda$ into itself.

The proof is complete.
\end{proof}

As an immediate consequence, we obtain the following fixed-point result.

\begin{corollary}[Existence and uniqueness in the weighted space]
\label{cor:weighted-fixed-point}
Suppose that Assumption~\ref{ass:weighted-contraction} holds, and let $\lambda>0$ satisfy $\lambda^\beta>L_f$. Then the operator $\mathcal{T}$ admits a unique fixed point $u\in\mathcal{X}_\lambda$, namely,
\[
u=\mathcal{T}u.
\]
Equivalently, the stochastic fixed-point equation \eqref{eq:nonlinear-fk}
has a unique solution in $\mathcal{X}_\lambda$.
\end{corollary}

\begin{proof}
By Theorem~\ref{thm:weighted-contraction-final}, $\mathcal{T}$ is a strict contraction on the Banach space $(\mathcal{X}_\lambda,\|\cdot\|_\lambda)$. The conclusion follows directly from the Banach fixed-point theorem.
\end{proof}

\begin{corollary}[Convergence of the relaxed Picard iteration]
\label{cor:relaxed-picard}
Suppose that Assumption~\ref{ass:weighted-contraction} holds, and let $\lambda>0$ and $\eta\in(0,1]$. Define
\begin{equation}\label{eq:relaxed-operator}
\mathcal{T}_\eta v := (1-\eta)v+\eta\mathcal{T}v,
\qquad v\in\mathcal{X}_\lambda.
\end{equation}
If $\lambda^\beta>L_f$, then $\mathcal{T}_\eta$ is a strict contraction on $\mathcal{X}_\lambda$ and has the same fixed point as $\mathcal{T}$. Consequently, the iteration
\begin{equation}\label{eq:relaxed-picard}
u^{(k+1)}=(1-\eta)u^{(k)}+\eta\mathcal{T}u^{(k)}
\end{equation}
converges in $\mathcal{X}_\lambda$ to the unique solution of \eqref{eq:nonlinear-fk}.
\end{corollary}

\begin{proof}
By Theorem~\ref{thm:weighted-contraction-final}, for all $v,w\in\mathcal{X}_\lambda$,
\[
\|\mathcal{T}v-\mathcal{T}w\|_\lambda
\le
\frac{L_f}{\lambda^\beta}\|v-w\|_\lambda.
\]
Hence
\[
\|\mathcal{T}_\eta v-\mathcal{T}_\eta w\|_\lambda
\le
(1-\eta)\|v-w\|_\lambda
+
\eta\|\mathcal{T}v-\mathcal{T}w\|_\lambda
\le
\left(1-\eta+\eta\frac{L_f}{\lambda^\beta}\right)\|v-w\|_\lambda.
\]
Since $\lambda^\beta>L_f$, the contraction factor is strictly less than $1$, and thus $\mathcal{T}_\eta$ is a strict contraction.

Finally, $\mathcal{T}_\eta$ and $\mathcal{T}$ clearly share the same fixed point. The convergence follows from the Banach fixed-point theorem.
\end{proof}

The contraction estimate of Theorem~\ref{thm:weighted-contraction-final} and
the resulting unique fixed point in Corollary~\ref{cor:weighted-fixed-point}
provide the analytical foundation for the Deep Picard scheme developed in
Section~\ref{subsec:method:dpi}, which realizes successive applications of
$\mathcal{T}$ through neural regression; the relaxed iteration of
Corollary~\ref{cor:relaxed-picard} additionally justifies the
relaxation-based stabilization built into the algorithm.

\section{Methodology}\label{sec:method}

This section turns \eqref{eq:nonlinear-fk} into an algorithm in two steps.
Section~\ref{subsec:sim-paths} explains how to simulate the underlying
trajectories, and Section~\ref{subsec:method:dpi} uses these trajectories to
train a neural network that approximates the fixed point of $\mathcal{T}$.

\subsection{Numerical simulation of the coupled space-time motion}
\label{subsec:sim-paths}

We describe the numerical approximation of the coupled stochastic processes
appearing in the nonlinear representation \eqref{eq:nonlinear-fk}, namely the
one-sided $\beta$-stable subordinator that drives the stochastic clock and the
rotationally symmetric $\alpha$-stable L\'evy process associated with the
fractional Laplacian. Throughout this section, we fix an operational-time step
size $\Delta s>0$ and focus on constructions that remain stable and efficient
in high spatial dimensions.

\subsubsection{Simulation of the $\beta$-stable subordinator}

To approximate the time-crossing event
\[
\tau_t=\inf\{s>0:\,Y_s^\beta\ge t\},
\]
we discretize the operational time by a uniform grid
\begin{equation}\label{eq:grid-rewrite}
0=s_0<s_1<s_2<\cdots,\qquad s_i=i\Delta s .
\end{equation}
Exploiting the stationary independent increments and self-similarity of stable
subordinators, the increments of $Y^\beta$ are generated recursively as
\begin{equation}\label{eq:Y-recursion-rewrite}
Y^\beta(s_i)
=
Y^\beta(s_{i-1})
+
(\Delta s)^{1/\beta}\,\eta_i,
\qquad
\eta_i\stackrel{\mathrm{i.i.d.}}{\sim} S_\beta^+,
\end{equation}
where $S_\beta^+$ denotes the positive $\beta$-stable law normalized by
\[
\mathbb{E}\!\left[e^{-\lambda\eta_i}\right]
=
e^{-\lambda^\beta},
\qquad \lambda>0.
\]
The discretized time-crossing index is then defined by
\begin{equation}\label{eq:tau-physical-rewrite}
i_\star(t)
=
\min\{i\ge0:\,Y^\beta(s_i)\ge t\},
\qquad
\tau_t\approx s_{i_\star(t)}=i_\star(t)\Delta s .
\end{equation}
The corresponding backward physical-time trajectory along the operational grid is
\[
t_{s_i}=t-Y^\beta(s_i).
\]
This construction provides a numerical mechanism for detecting the event
$\{t_s\le0\}$ and is coupled with the spatial exit event through the joint
stopping time
\[
\tau(t,x)=\tau_t\wedge\tau_\Omega(x).
\]

\subsubsection{Walk-on-spheres simulation of the rotationally symmetric $\alpha$-stable process}

We next approximate the spatial motion of the $\alpha$-stable process using a
walk-on-spheres-type strategy adapted to fractional Laplacian operators
\cite{KyprianouOsojnikShardlow2017UnbiasedWoS,ShengSuXu2023EfficientMCIFL}.
Starting from the current position, we choose a ball radius so that the mean
exit time of the $\alpha$-stable process from the ball matches the
operational-time increment $\Delta s$. The exit location becomes the next
center of the spatial trajectory, and the procedure is repeated until a
stopping event is triggered.

Let $c\in\mathbb{R}^d$ and denote by
\[
B_r(c)=\{y\in\mathbb{R}^d:\|y-c\|\le r\}
\]
the ball of radius $r>0$ centered at $c$. For the centered $\alpha$-stable
process $X^\alpha$ with $X^\alpha(0)=0$, define
\begin{equation}\label{eq:first-exit-rewrite}
\tau_r
:=
\inf\{s>0:\,X^\alpha(s)\notin B_r(0)\}.
\end{equation}
The first two moments of $\tau_r$ are available in closed form
\cite{getoor1961first}:
\begin{equation}\label{eq:exit-moments-rewrite}
\mathbb{E}[\tau_r]=r^\alpha \kappa_{d,\alpha},
\qquad
\mathbb{E}[\tau_r^2]
=
\alpha r^\alpha \kappa_{d,\alpha}^{\,2}
\int_{0}^{r^2}
\nu^{\frac{\alpha}{2}-1}\,
{}_2F_1\!\left(
-\frac{\alpha}{2},\frac{d}{2};
\frac{d+\alpha}{2};
\nu r^{-2}
\right)\,\mathrm{d}\nu,
\end{equation}
where
\[
\kappa_{d,\alpha}
=
\frac{\Gamma(\frac{d}{2})}
{2^\alpha\Gamma(1+\frac{\alpha}{2})\Gamma(\frac{d+\alpha}{2})}.
\]
Here ${}_2F_1$ denotes the Gauss hypergeometric function. Its Euler-type
integral representation is
\begin{equation}\label{eq:2F1-def-rewrite}
{}_2F_1(a,b;c;z)
=
\frac{\Gamma(c)}{\Gamma(b)\Gamma(c-b)}
\int_0^1
\xi^{b-1}(1-\xi)^{c-b-1}(1-z\xi)^{-a}\,\mathrm{d}\xi .
\end{equation}
Moreover, Chebyshev's inequality together with
\eqref{eq:exit-moments-rewrite} implies that, for any $\varepsilon>0$,
\begin{equation}\label{eq:exit-prob-rewrite}
\mathbb{P}\!\left(
\big|\tau_r-\mathbb{E}[\tau_r]\big|\le \varepsilon
\right)
\ge
1-\frac{r^{2\alpha}\kappa_{d,\alpha}^{\,2}}{\varepsilon^2},
\qquad r\downarrow0.
\end{equation}
Thus $\tau_r$ concentrates around its mean in the small-radius regime.

Matching the expected exit time with the operational step size,
$\mathbb{E}[\tau_r]=\Delta s$, yields the radius
\begin{equation}\label{eq:radius-ball}
r
=
\left(
\frac{\Delta s}{\kappa_{d,\alpha}}
\right)^{1/\alpha}.
\end{equation}
Hence each spatial update represents, on average, an operational-time advance
of length $\Delta s$.

It remains to sample the exit location from the ball. By rotational symmetry,
we write the exit point relative to the current center as $J\vartheta$, where
$\vartheta\sim\mathrm{Unif}(\mathbb{S}^{d-1})$ and $J>r$ is the radial exit
distance. For a process starting from the center of $B_r(0)$, the radial
distribution satisfies, for $\gamma>r$,
\begin{equation}\label{eq:jump-cdf-rewrite}
\mathbb{P}\bigl(r<|X^\alpha_{\tau_r}|<\gamma\bigr)
=
\frac{\sin(\pi\alpha/2)}{\pi}
\left[
\mathrm{B}\!\left(1-\frac{\alpha}{2},\frac{\alpha}{2}\right)
-
\mathrm{B}\!\left(
\frac{r^2}{\gamma^2};
1-\frac{\alpha}{2},\frac{\alpha}{2}
\right)
\right],
\end{equation}
where $\mathrm{B}(z;a,b)$ denotes the incomplete Beta function. Therefore, if
$\omega\sim\mathrm{Unif}(0,1)$, inverse-transform sampling gives
\begin{equation}\label{eq:jump-distance-rewrite}
J
=
\frac{r}{
\sqrt{
\mathrm{B}^{-1}\!\left(
\mathrm{B}\!\left(1-\frac{\alpha}{2},\frac{\alpha}{2}\right)
-\frac{\pi\omega}{\sin(\pi\alpha/2)}
\ ;\
1-\frac{\alpha}{2},\frac{\alpha}{2}
\right)
}},
\end{equation}
where $\mathrm{B}^{-1}(\cdot\,;a,b)$ denotes the inverse of
$z\mapsto \mathrm{B}(z;a,b)$. The numerical spatial chain is then updated by
\begin{equation}\label{eq:wos-update}
X_i=X_{i-1}+J_i\,\vartheta_i,
\qquad X_i\approx X^\alpha_{s_i},
\end{equation}
with i.i.d.\ directions $\vartheta_i\sim\mathrm{Unif}(\mathbb{S}^{d-1})$ and
jump lengths $J_i$ generated by \eqref{eq:jump-distance-rewrite}. In practice,
the directions can be sampled by drawing $G_i\sim\mathcal{N}(0,I_d)$ and
setting $\vartheta_i=G_i/\|G_i\|$.

Altogether, the above procedure yields a mesh-free trajectory simulator suitable
for high-dimensional coupled space-time fractional dynamics.

\subsection{Deep Picard iteration for space-time fractional diffusion PDEs}
\label{subsec:method:dpi}

We approximate the solution of the nonlinear space-time fractional problem
\eqref{eq:intro-pde} by learning the fixed point of the nonlinear
Feynman--Kac operator induced by the coupled processes $(X_s^x,Y_s^\beta)$.
For any measurable function $v:[0,T]\times\overline{\Omega}\to\mathbb{R}$,
define
\begin{equation}\label{eq:T-operator-dpi}
(\mathcal{T}v)(t,x)
:=
\mathbb{E}_x\!\left[
\mathcal{G}(t,x)
+
\int_{0}^{\tau(t,x)}
f\!\bigl(t_s, v(t_s, X_s^x), X_s^x\bigr)\,\rd s
\right],
\end{equation}
where the physical time along operational time $s$ is
\[
t_s=t-Y_s^\beta,
\]
and the stopping time is
\[
\tau(t,x):=\tau_t\wedge\tau_\Omega(x),
\qquad
\tau_t=\inf\{s>0:\,Y_s^\beta\ge t\},
\qquad
\tau_\Omega(x)=\inf\{s>0:\,X_s^x\notin\Omega\}.
\]

The terminal payoff $\mathcal{G}(t,x)$ encodes whether the stochastic trajectory
reaches the initial-time surface first or exits the spatial domain first. It is
given by
\begin{equation}\label{eq:terminal-payoff}
\mathcal{G}(t,x)
=
u_0\!\bigl(X_{\tau_t}^x\bigr)
\mathbf{1}_{\{\tau_t<\tau_\Omega(x)\}}
+
g\!\bigl(t_{\tau_\Omega(x)},X_{\tau_\Omega(x)}^x\bigr)
\mathbf{1}_{\{\tau_\Omega(x)\le\tau_t\}},
\end{equation}
where
\[
t_{\tau_\Omega(x)}:=t-Y_{\tau_\Omega(x)}^\beta
\]
is the physical time associated with the spatial exit point. The first
indicator corresponds to trajectories that remain in $\Omega$ until the
stochastic clock satisfies $t_s\le0$; the second corresponds to trajectories
that exit $\Omega$ at positive physical time before the clock reaches zero.

The PDE solution $u$ satisfies the fixed-point relation $u=\mathcal{T}u$. We
construct Picard iterates $\{u^{(k)}\}_{k\ge0}$ by
\begin{equation}\label{eq:picard-functional-dpi}
u^{(k+1)}=\mathcal{T}u^{(k)},
\qquad k=0,1,2,\ldots,
\end{equation}
initialized by $u^{(0)}\equiv0$.

For numerical realization, each iterate is parameterized by a neural network
$u_{\theta_k}(t,x)$. At Picard level $k$, we sample training inputs
$\{(t_i,x_i)\}_{i=1}^N$ from $(0,T]\times\Omega$. For each input, we simulate
$M$ independent trajectories of $(X_s^{x_i},Y_s^\beta)$ on the operational-time
grid
\[
s_\ell=\ell\Delta s,\qquad \ell=0,1,2,\ldots .
\]
The subordinator is advanced using the recursion
\eqref{eq:Y-recursion-rewrite}. The spatial process performs one fractional
walk-on-spheres update per operational step, with radius determined by the
mean-exit-time matching rule
\[
r
=
\left(
\frac{\Delta s}{\kappa_{d,\alpha}}
\right)^{1/\alpha}.
\]

For the $m$-th trajectory starting from $(t_i,x_i)$, define
\[
t_{i,\ell}^{(m)}
:=
t_i-Y_{s_\ell}^{\beta,(m)},
\qquad
x_{i,\ell}^{(m)}
\approx
X_{s_\ell}^{x_i,(m)}.
\]
The trajectory is stopped at the first index $\ell_\star^{(m)}$ such that
\[
t_{i,\ell_\star^{(m)}}^{(m)}\le0
\qquad\text{or}\qquad
x_{i,\ell_\star^{(m)}}^{(m)}\notin\Omega .
\]
The terminal contribution $\widehat{\mathcal{G}}_i^{(m)}$ is evaluated at this
stopping index according to \eqref{eq:terminal-payoff}. Let
\[
N_s^{(m)}:=\ell_\star^{(m)}-1
\]
denote the last accepted pre-stopping index. The path integral is then
approximated only over the accepted pre-stopping states.

For the integral term, a standard numerical quadrature is applied along the
accepted pre-stopping states. Writing
\[
f_{i,\ell}^{(m)}
:=
f\!\bigl(
t_{i,\ell}^{(m)},
u_{\theta_k}(t_{i,\ell}^{(m)},x_{i,\ell}^{(m)}),
x_{i,\ell}^{(m)}
\bigr),
\]
with the convention
\[
(t_{i,0}^{(m)},x_{i,0}^{(m)})=(t_i,x_i),
\]
we consider two quadrature rules. The right-point rectangle rule is
\begin{equation}\label{eq:quad-rect}
\int_0^{\tau(t_i,x_i)}
f\!\bigl(t_s,u_{\theta_k}(t_s,X_s),X_s\bigr)\,\rd s
\approx
\sum_{\ell=1}^{N_s^{(m)}}
f_{i,\ell}^{(m)}\,\Delta s ,
\end{equation}
while the trapezoidal rule is
\begin{equation}\label{eq:quad-trap}
\int_0^{\tau(t_i,x_i)}
f\!\bigl(t_s,u_{\theta_k}(t_s,X_s),X_s\bigr)\,\rd s
\approx
\sum_{\ell=1}^{N_s^{(m)}}
\frac{
f_{i,\ell-1}^{(m)}+f_{i,\ell}^{(m)}
}{2}\,\Delta s .
\end{equation}
For smooth deterministic integrands, the rectangle and trapezoidal rules are
formally first- and second-order accurate, respectively. Along the present
jump-driven stochastic trajectories, however, the practical accuracy is also
affected by path discretization and Monte Carlo variance. Our implementation
uses \eqref{eq:quad-trap}; a numerical comparison with \eqref{eq:quad-rect} is
given in Section~\ref{subsec:disk}.

The Monte Carlo regression label for input $(t_i,x_i)$ is
\begin{equation}\label{eq:mc-label-dpi}
\widehat{y}_i^{(k)}
=
\frac{1}{M}
\sum_{m=1}^M
\left[
\widehat{\mathcal{G}}_i^{(m)}
+
\sum_{\ell=1}^{N_s^{(m)}}
\frac{
f_{i,\ell-1}^{(m)}+f_{i,\ell}^{(m)}
}{2}\,\Delta s
\right].
\end{equation}

Finally, the next Picard network parameters are obtained by solving
\begin{equation}\label{eq:supervised-loss-dpi}
\theta_{k+1}
\in
\argmin_{\theta}
\frac{1}{N}
\sum_{i=1}^N
\bigl|u_{\theta}(t_i,x_i)-\widehat{y}_i^{(k)}\bigr|^2,
\end{equation}
which trains $u_{\theta_{k+1}}$ to regress the Monte Carlo approximation of
$\mathcal{T}u_{\theta_k}$.

Iterating \eqref{eq:mc-label-dpi}--\eqref{eq:supervised-loss-dpi} yields a
sequence of neural approximations to the fixed point of $\mathcal{T}$, and
therefore to the solution of the nonlinear fractional PDE.

\paragraph{Relaxed Picard variant.}
To stabilize the iteration when the neural approximation is still far from the
fixed point, we follow Corollary~\ref{cor:relaxed-picard} and consider the
relaxed update
\begin{equation}\label{eq:relaxed-picard-dpi}
u^{(k+1)}
=
(1-\eta)u^{(k)}+\eta\mathcal{T}u^{(k)},
\qquad \eta\in(0,1],
\end{equation}
which shares the unique fixed point of $\mathcal{T}$. At the network level,
this amounts to replacing the regression target in
\eqref{eq:supervised-loss-dpi} with the convex combination
\begin{equation}\label{eq:relaxed-mc-label}
\widehat{y}_i^{(k,\eta)}
:=
(1-\eta)u_{\theta_k}(t_i,x_i)
+
\eta\widehat{y}_i^{(k)}.
\end{equation}
Training $u_{\theta_{k+1}}$ against $\widehat{y}_i^{(k,\eta)}$ gives the
relaxed Picard update. Setting $\eta=1$ recovers the standard Picard scheme,
whereas $\eta\in(0,1)$ damps the update and improves stability when the
nonlinear term amplifies Monte Carlo or regression errors. The empirical effect
of $\eta$ is examined in Section~\ref{sec:numexp}.

\begin{algorithm}[t]
\caption{Deep Picard iteration for nonlinear space-time fractional diffusion PDEs}
\label{alg:dpi-stfd}
\begin{algorithmic}[1]
\Require Picard iterations $K$, dataset size $N$, MC paths per point $M$, operational step $\Delta s$, relaxation parameter $\eta\in(0,1]$.
\State Initialize $u_{\theta_0}\equiv0$ on $(0,T]\times\Omega$.
\For{$k=0,1,\ldots,K-1$}
  \State Sample training inputs $\{(t_i,x_i)\}_{i=1}^N\subset(0,T]\times\Omega$.
  \For{$i=1,\ldots,N$}
    \State Generate $M$ i.i.d.\ coupled trajectories
    $\{(X^{(m)}_{s_\ell},Y^{\beta,(m)}_{s_\ell})\}_{\ell\ge0}$
    on $s_\ell=\ell\Delta s$ using \eqref{eq:Y-recursion-rewrite} and the
    WoS sampler with radius \eqref{eq:radius-ball}, stopped at
    $\tau^{(m)}=\tau(t_i,x_i)$.
    \State For each path, evaluate the terminal payoff
    $\widehat{\mathcal{G}}_i^{(m)}$ at the stopping index and approximate the
    path integral by the trapezoidal rule \eqref{eq:quad-trap}.
    \State Compute the MC label
    \[
    \widehat y_i^{(k)}
    =
    \frac{1}{M}\sum_{m=1}^M
    \left[
    \widehat{\mathcal{G}}_i^{(m)}
    +
    \sum_{\ell=1}^{N_s^{(m)}}
    \frac{
    f_{i,\ell-1}^{(m)}+f_{i,\ell}^{(m)}
    }{2}\Delta s
    \right].
    \]
    \State Form the regression target
    \[
    \widehat{y}_i^{(k,\eta)}
    =
    (1-\eta)u_{\theta_k}(t_i,x_i)+\eta\widehat{y}_i^{(k)}.
    \]
  \EndFor
  \State Train $u_{\theta_{k+1}}$ by supervised regression of
  $\{\widehat{y}_i^{(k,\eta)}\}_{i=1}^N$ via \eqref{eq:supervised-loss-dpi}.
\EndFor
\State \Return $u_{\theta_K}$.
\end{algorithmic}
\end{algorithm}

\section{Numerical Experiments}\label{sec:numexp}

We evaluate Algorithm~\ref{alg:dpi-stfd} on a sequence of benchmark problems
of increasing difficulty, each equipped with a known reference solution.
Section~\ref{subsec:disk} considers the two-dimensional unit disk, on which
the fractional Laplacian of the spatial profile admits a closed form; this
benchmark serves to calibrate the method, study its convergence with respect
to the discretization parameters, and isolate the effect of Picard
relaxation. Section~\ref{subsec:rect} replaces the disk by the unit square
to test the algorithm on a non-smooth boundary. Section~\ref{subsec:doublebump}
returns to the disk with a multi-modal double-bump profile that probes the
ability of the network to resolve spatially heterogeneous solutions.
Section~\ref{subsec:highdim} concludes with $d=20,50,100$ unit-ball
experiments that examine the scalability of the method with the spatial
dimension. On every domain, two reaction nonlinearities of increasing
severity are tested.

\subsection{Experiment setup}\label{subsec:setup}

Throughout this section we fix the fractional orders $\alpha=1.5$ and $\beta=0.6$
and the time horizon $T=1.0$. To enforce the homogeneous Dirichlet condition
exactly, the neural approximation is parameterized in the hard-constrained
form
\begin{equation}\label{eq:hard-constraint}
u_\theta(t,x)=\varphi(x)\,v_\theta(t,x),
\end{equation}
where $\varphi$ is a smooth boundary factor satisfying $\varphi|_{\partial\Omega}=0$
and $v_\theta$ is a fully connected residual network with $\tanh$ activation.
This construction removes the boundary error from the optimization and lets the
training focus on the interior dynamics. The two-dimensional experiments
employ $3$ residual blocks of width $128$; the high-dimensional experiments
use $4$--$5$ blocks of width $512$--$1024$, with the precise configuration for
each $d$ given in Section~\ref{subsec:highdim}.

At each Picard iteration $k$, $N$ collocation points are sampled uniformly in
$(0,T]\times\Omega$ and $M$ independent Monte Carlo trajectories per point are
generated by the coupled subordinator--WoS simulator of
Section~\ref{subsec:sim-paths} with operational time step $\Delta s$. The
regression labels $\{\widehat{y}_i^{(k)}\}$ are computed via
\eqref{eq:mc-label-dpi} and combined into the relaxed target
$\widehat{y}_i^{(k,\eta)}$ as in Algorithm~\ref{alg:dpi-stfd}; the case
$\eta=1$ recovers the standard Picard update. Each Picard step solves the
supervised regression \eqref{eq:supervised-loss-dpi} by Adam with learning
rate $10^{-4}$, mini-batch size $512$, and $8{,}000$ gradient steps,
warm-started from $\theta_k$.

Accuracy is measured by two complementary RMSE metrics. The \emph{slice RMSE}
quantifies the spatial error at the terminal time on a uniform
$201\times 201$ grid $\mathcal{G}$ restricted to $\Omega$,
\begin{equation}\label{eq:slice-rmse}
\mathrm{RMSE}_{\mathrm{slice}}
=
\left(\frac{1}{|\mathcal{G}|}\sum_{x_j\in\mathcal{G}}
\bigl|u_{\theta_k}(T,x_j)-u_{\mathrm{ex}}(T,x_j)\bigr|^2\right)^{1/2},
\end{equation}
while the \emph{spacetime RMSE} measures the global accuracy on a freshly
sampled spacetime ensemble,
\begin{equation}\label{eq:spacetime-rmse}
\mathrm{RMSE}_{\mathrm{st}}
=
\left(\frac{1}{10000}\sum_{l=1}^{10}\sum_{j=1}^{1000}
\bigl|u_{\theta_k}(t_l,x_j^{(l)})-u_{\mathrm{ex}}(t_l,x_j^{(l)})\bigr|^2\right)^{1/2},
\end{equation}
where $\{t_l\}_{l=1}^{10}$ are equally spaced in $[0.1,1.0]$ and, for each
$t_l$, $\{x_j^{(l)}\}_{j=1}^{1000}$ are drawn independently and uniformly in
$\Omega$. Unless stated otherwise, the two-dimensional experiments use the
default configuration $N=32{,}768$ and $M=16$, while the operational
step $\Delta s$ and the relaxation parameter $\eta$ are specified per example.
All experiments were conducted on a single NVIDIA RTX 4090 GPU; quadrature and
precomputation details specific to individual examples are reported in the
Appendix.

\subsection{Example 1: Unit disk}\label{subsec:disk}

We begin with the unit disk $\Omega=B_1(0)\subset\mathbb{R}^2$, on which the
fractional Laplacian of the spatial profile admits a closed form. This
benchmark removes any precomputation error and isolates the algorithmic
behaviour of the Deep Picard iteration.

We consider the nonlinear space-time fractional Dirichlet problem
\begin{equation}\label{eq:stfd-nl-system}
\begin{cases}
\partial_t^\beta u + (-\Delta)^{\alpha/2} u = f(t,u,x), & (t,x)\in(0,T]\times\Omega,\\
u(t,x)=0, & (t,x)\in[0,T]\times\Omega^c,\\
u(0,x)=0, & x\in\Omega,
\end{cases}
\end{equation}
with reference exact solution
\begin{equation}\label{eq:exact-sol-disk}
u_{\mathrm{ex}}(t,x) = t^\beta\,\varphi(x), \qquad \varphi(x) := (1 - |x|^2)_+^{\alpha/2}.
\end{equation}
The profile $\varphi$ satisfies $(-\Delta)^{\alpha/2}\varphi(x) = C_{2,\alpha}$
for all $x\in\Omega$ with
$C_{2,\alpha} = 2^\alpha\bigl(\Gamma(1+\alpha/2)\bigr)^2$, yielding a
closed-form fractional Laplacian. Two reaction nonlinearities of increasing
severity are tested:
\begin{itemize}
\item \textbf{Setting~A} (quadratic): $f_A(t,u,x) = \lambda u^2 + g_A(t,x)$, with $\lambda=1$ and
\[
g_A(t,x) = \Gamma(\beta{+}1)\,\varphi(x) + t^\beta C_{2,\alpha} - \lambda\,t^{2\beta}\varphi(x)^2.
\]
\item \textbf{Setting~B} (cubic): $f_B(t,u,x) = \kappa(u - u^3) + g_B(t,x)$, with $\kappa=1$ and
\[
g_B(t,x) = \Gamma(\beta{+}1)\,\varphi(x) + t^\beta C_{2,\alpha} - \kappa\bigl(t^\beta\varphi(x) - t^{3\beta}\varphi(x)^3\bigr).
\]
\end{itemize}
Setting~B is the more demanding of the two: its Jacobian
$|\partial_u f_B|=|\kappa(1-3u^2)|$ can reach $3\kappa$, amplifying Monte
Carlo label noise across successive Picard iterations.  We run Algorithm~\ref{alg:dpi-stfd} at $\Delta s=5\times 10^{-3}$ and compare
$\eta=1.0$ against $\eta=0.6$, with all other parameters at their default
values.

Table~\ref{tab:disk-results} reports the resulting errors. Relaxation reduces the spacetime RMSE by $34\%$ for Setting~A and
by $31\%$ for Setting~B. The mechanism is consistent across the two
nonlinearities: without relaxation the iteration reaches a transient accuracy
plateau and then oscillates as Monte Carlo label noise accumulates, whereas
damping the update with $\eta=0.6$ trades a slower initial transient for a
lower and stabler terminal error. The benefit
is larger for Setting~B, whose Jacobian $|\partial_u f_B|\le 3\kappa$
amplifies the label variance. As shown in
Figure~\ref{fig:circle-heatmap}, the spatial profile is recovered with high
fidelity; the residual error concentrates near $\partial\Omega$, where the
boundary singularity of $\varphi$ steepens the spatial gradient.

\begin{table}[pos=htbp]
\centering
\caption{Effect of Picard relaxation on the unit disk ($M=16$, $\Delta s = 5\times 10^{-3}$). Bold values indicate the best result for each setting.}
\label{tab:disk-results}
\begin{tabular}{llcc}
\toprule
Setting & $\eta$ & Slice RMSE & Spacetime RMSE \\
\midrule
A & $1.0$ & $1.88\times 10^{-2}$ & $1.21\times 10^{-2}$ \\
A & $0.6$ & $\mathbf{1.16\times 10^{-2}}$ & $\mathbf{8.00\times 10^{-3}}$ \\
\midrule
B & $1.0$ & $1.90\times 10^{-2}$ & $1.23\times 10^{-2}$ \\
B & $0.6$ & $\mathbf{1.18\times 10^{-2}}$ & $\mathbf{8.50\times 10^{-3}}$ \\
\bottomrule
\end{tabular}
\end{table}

\begin{figure}[pos=htbp]
\centering
\subcaptionbox{Setting~A: predicted (left), exact (center), error (right) at $t=1.0$}{
\includegraphics[width=0.78\linewidth]{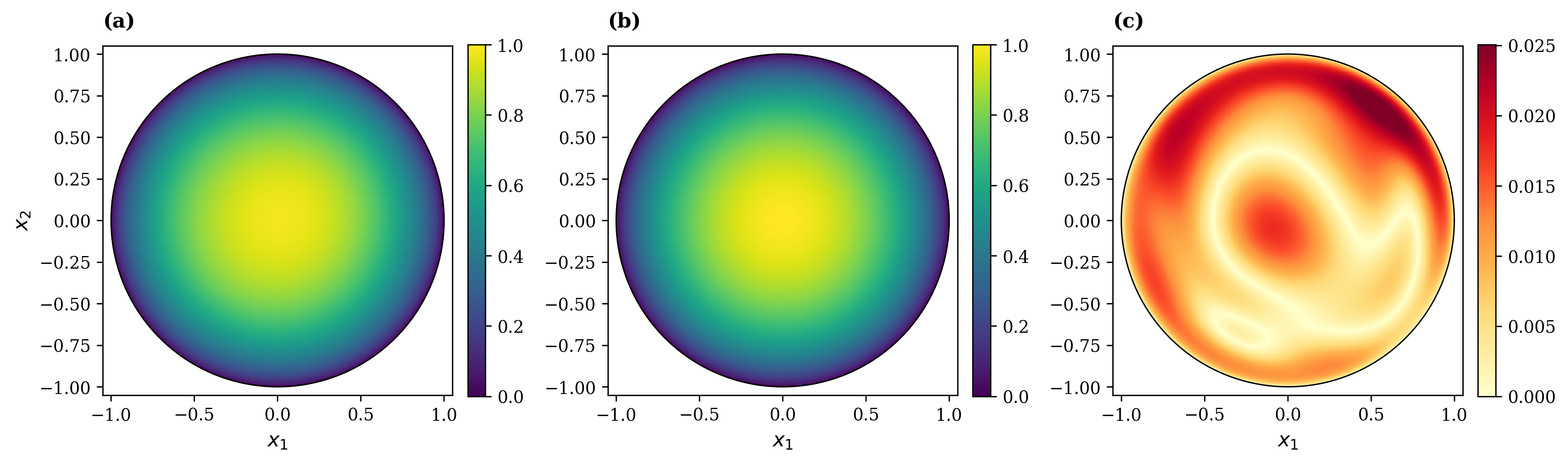}
}
\subcaptionbox{Setting~B: predicted (left), exact (center), error (right) at $t=1.0$}{
\includegraphics[width=0.78\linewidth]{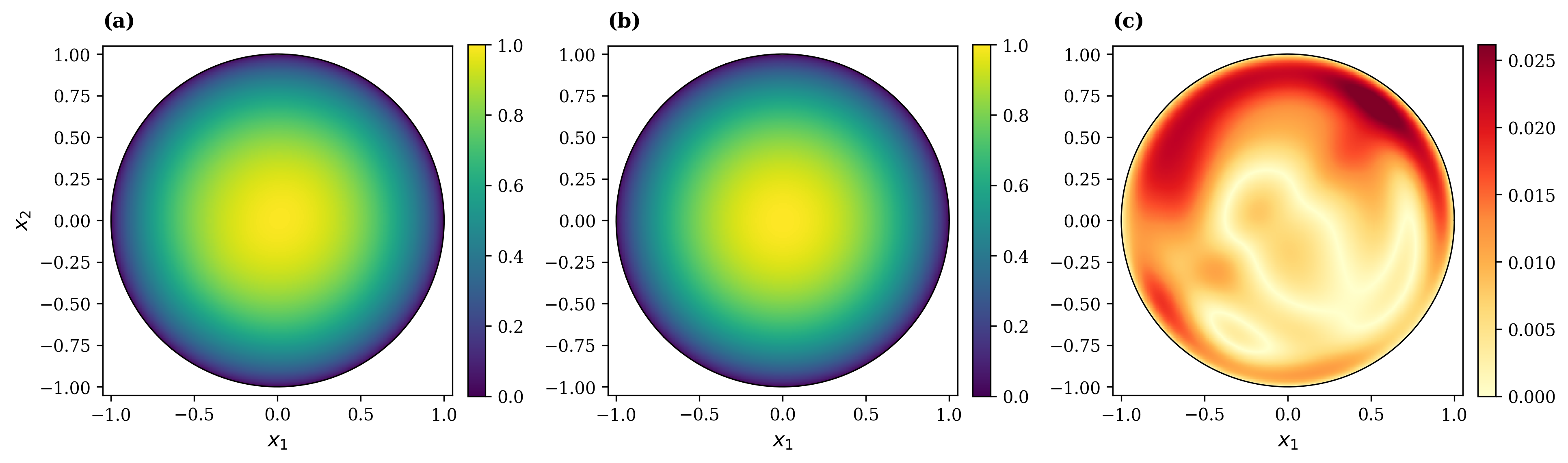}
}
\caption{Solution quality at $t=T=1.0$ for the unit disk with $\eta=0.6$.}
\label{fig:circle-heatmap}
\end{figure}

We next study how the error depends on the operational step $\Delta s$, the
number of Monte Carlo paths $M$, and the choice of quadrature rule, all
under Setting~A with $\eta=0.6$.

\paragraph{Step-size convergence.}
Refining $\Delta s$ over $\{4\times 10^{-2},2\times 10^{-2},10^{-2},5\times 10^{-3}\}$
reduces the spacetime RMSE by roughly a fivefold factor, see
Table~\ref{tab:conv-ds} and Figure~\ref{fig:conv-order}(a). A least-squares
fit of $\log\mathrm{RMSE}_{\mathrm{st}}$ against $\log\Delta s$ yields an
empirical rate of approximately $0.81$, indicating sublinear convergence of
the time discretization at the resolutions tested.

\paragraph{Monte Carlo path convergence.}
With $\Delta s$ fixed at $2\times 10^{-3}$, varying $M\in\{4,8,16,32\}$
yields the spacetime RMSE reported in Table~\ref{tab:conv-mc} and
Figure~\ref{fig:conv-order}(b). The error decays at a rate close to the
theoretical $\mathcal{O}(M^{-1/2})$ between $M=4$ and $M=16$ and then
saturates at $M=32$, signalling that the time-step discretization error
has become the dominant component at the accuracy level
${\sim}\,6\times 10^{-3}$.

\begin{table}[pos=htbp]
\centering
\caption{Accuracy under different step sizes on the unit disk (Setting~A, $M=16$, $\eta=0.6$).}
\label{tab:conv-ds}
\begin{tabular}{ccc}
\toprule
$\Delta s$ & Slice RMSE & Spacetime RMSE \\
\midrule
$4.0\times 10^{-2}$ & $5.60\times 10^{-2}$ & $4.37\times 10^{-2}$ \\
$2.0\times 10^{-2}$ & $3.21\times 10^{-2}$ & $2.51\times 10^{-2}$ \\
$1.0\times 10^{-2}$ & $2.16\times 10^{-2}$ & $1.57\times 10^{-2}$ \\
$5.0\times 10^{-3}$ & $1.16\times 10^{-2}$ & $8.00\times 10^{-3}$ \\
\bottomrule
\end{tabular}
\end{table}

\begin{table}[pos=htbp]
\centering
\caption{Monte Carlo path convergence on the unit disk (Setting~A, $\Delta s = 2\times 10^{-3}$, $\eta=0.6$).}
\label{tab:conv-mc}
\begin{tabular}{ccc}
\toprule
$M$ & Slice RMSE & Spacetime RMSE \\
\midrule
$4$  & $2.10\times 10^{-2}$ & $1.12\times 10^{-2}$ \\
$8$  & $1.57\times 10^{-2}$ & $1.01\times 10^{-2}$ \\
$16$ & $9.66\times 10^{-3}$ & $\mathbf{6.09\times 10^{-3}}$ \\
$32$ & $\mathbf{7.21\times 10^{-3}}$ & $6.81\times 10^{-3}$ \\
\bottomrule
\end{tabular}
\end{table}

\paragraph{Effect of the quadrature rule.}
We finally compare the right-point rectangle rule \eqref{eq:quad-rect}
against the trapezoidal rule \eqref{eq:quad-trap} under identical seeds and
network configuration. As reported in Table~\ref{tab:quad-comparison}, the
two quadratures yield spacetime RMSEs that differ by less than $10\%$ in
every row and neither rule dominates uniformly. The quadrature discretization
is therefore not the dominant error source; the Monte Carlo variance and the
neural regression residual are. We adopt the trapezoidal rule throughout the
remaining experiments out of convention.

\begin{table}[pos=htbp]
\centering
\caption{Comparison of quadrature rules on the unit disk (Setting~A, $M=16$, $\eta=0.6$). Bold values mark the smaller spacetime RMSE in each row.}
\label{tab:quad-comparison}
\begin{tabular}{ccc}
\toprule
$\Delta s$ & Rectangle RMSE & Trapezoidal RMSE \\
\midrule
$4.0\times 10^{-2}$ & $4.37\times 10^{-2}$ & $\mathbf{4.07\times 10^{-2}}$ \\
$2.0\times 10^{-2}$ & $\mathbf{2.51\times 10^{-2}}$ & $2.55\times 10^{-2}$ \\
$1.0\times 10^{-2}$ & $\mathbf{1.57\times 10^{-2}}$ & $1.62\times 10^{-2}$ \\
$5.0\times 10^{-3}$ & $\mathbf{8.00\times 10^{-3}}$ & $8.61\times 10^{-3}$ \\
\bottomrule
\end{tabular}
\end{table}

\begin{figure}[pos=htbp]
\centering
\includegraphics[width=0.95\linewidth]{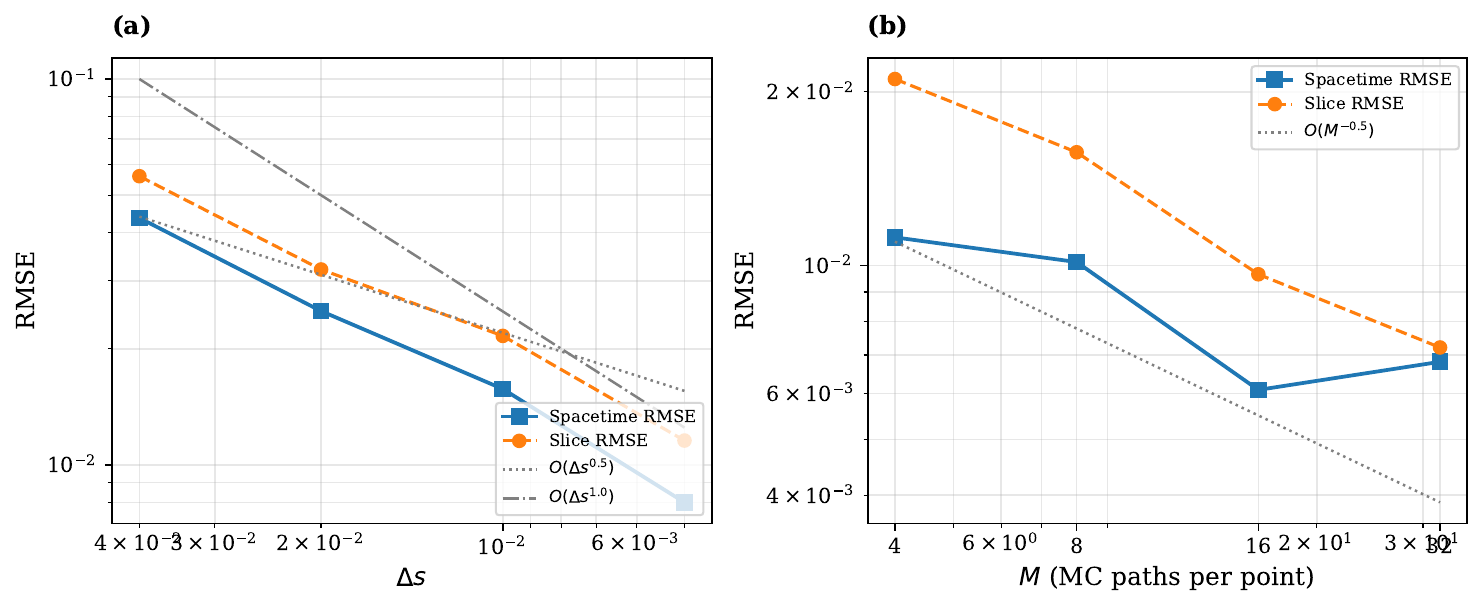}
\caption{Convergence of the spacetime RMSE on the unit disk (Setting~A, $\eta=0.6$).
(a)~Step-size convergence with $M=16$ fixed.
(b)~Monte Carlo path convergence with $\Delta s = 2\times 10^{-3}$ fixed.}
\label{fig:conv-order}
\end{figure}

\subsection{Example 2: Unit square}\label{subsec:rect}

We next replace the disk by the unit square $\Omega=(-1,1)^2$, which probes
the algorithm on a non-smooth boundary and on a profile whose fractional
Laplacian no longer admits a closed form. The reference solution is
$u_{\mathrm{ex}}(t,x)=t^\beta\,\varphi_R(x)$ with
\begin{equation}\label{eq:phi-rect}
\varphi_R(x) = (1-x_1^2)_+^2\,(1-x_2^2)_+^2,
\end{equation}
and the consistency forcing is constructed from $u_{\mathrm{ex}}$ in the
same way as in Example~\ref{subsec:disk}, with
$q_{\alpha,R}(x):=(-\Delta)^{\alpha/2}\varphi_R(x)$ in place of the
analytical constant $C_{2,\alpha}$. Since $\varphi_R$ is compactly supported,
$q_{\alpha,R}$ admits the Fourier-multiplier representation
$\mathcal{F}^{-1}\bigl(|\xi|^\alpha\widehat{\varphi_R}\bigr)$ and is
precomputed by a single 2D FFT on a periodic grid containing the support of
$\varphi_R$ prior to training. Two reaction nonlinearities of the same form
as before are tested:
\begin{itemize}
\item \textbf{Setting~A}: $f_A(t,u,x) = \lambda u^2 + \Gamma(\beta{+}1)\,\varphi_R(x) + t^\beta\,q_{\alpha,R}(x) - \lambda\,t^{2\beta}\varphi_R(x)^2$.
\item \textbf{Setting~B}: $f_B(t,u,x) = \kappa(u-u^3) + \Gamma(\beta{+}1)\,\varphi_R(x) + t^\beta\,q_{\alpha,R}(x) - \kappa\bigl(t^\beta\varphi_R(x) - t^{3\beta}\varphi_R(x)^3\bigr)$.
\end{itemize}
We run both settings at $\Delta s=2\times 10^{-3}$ and $\eta=0.5$.

Table~\ref{tab:rect-results} reports the resulting errors. Both settings exhibit monotone convergence
(Figure~\ref{fig:rect-convergence}); Setting~B even attains the lowest
spacetime RMSE among all two-dimensional experiments, indicating that the
choice $\eta=0.5$ supplies sufficient variance control for both
nonlinearities on this domain. The smaller errors compared with Example~\ref{subsec:disk} are attributable
to the higher boundary regularity of $\varphi_R$ on the square relative to
$\varphi$ on the disk. As shown in Figure~\ref{fig:rect-heatmap}, the
residual error concentrates at the corners, where $\varphi_R$ varies most
rapidly along the diagonal direction.

\begin{table}[pos=htbp]
\centering
\caption{Error metrics for the unit square examples ($\eta=0.5$, $M=16$).}
\label{tab:rect-results}
\begin{tabular}{lcc}
\toprule
Setting & Slice RMSE & Spacetime RMSE \\
\midrule
A ($\lambda u^2$) & $6.98\times 10^{-3}$ & $3.66\times 10^{-3}$ \\
B ($\kappa(u-u^3)$) & $\mathbf{4.33\times 10^{-3}}$ & $\mathbf{2.86\times 10^{-3}}$ \\
\bottomrule
\end{tabular}
\end{table}

\begin{figure}[pos=htbp]
\centering
\includegraphics[width=0.95\linewidth]{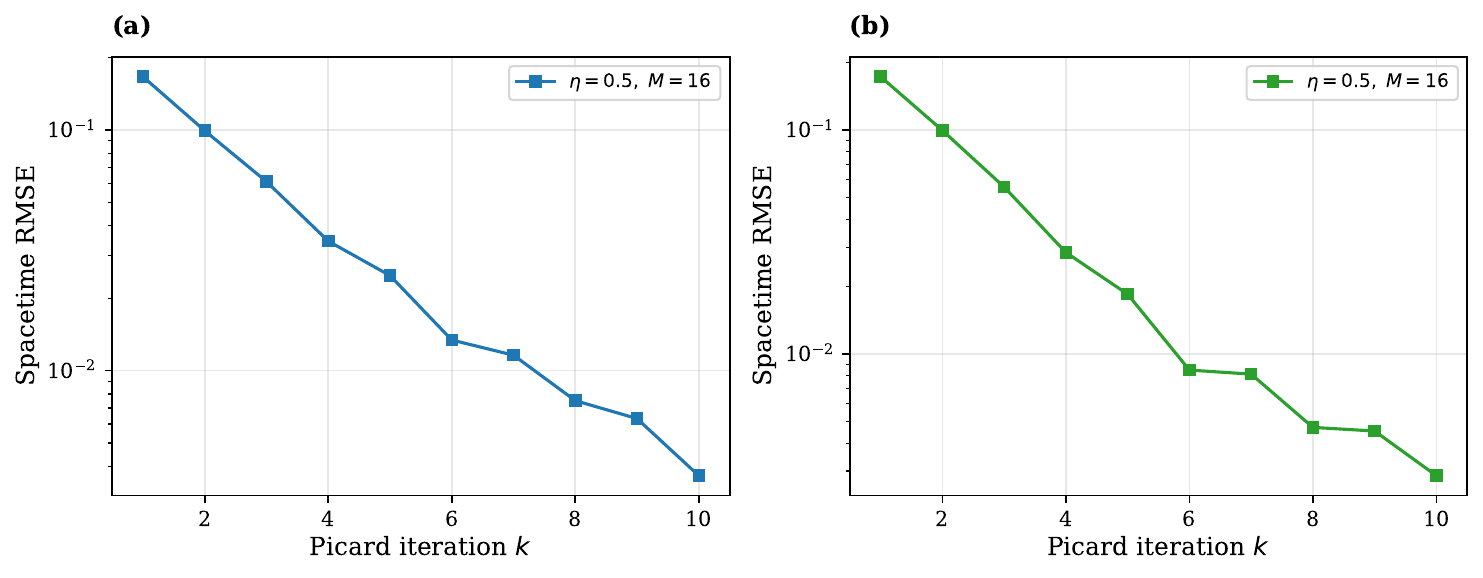}
\caption{Spacetime RMSE versus Picard iteration for the unit square examples.  (a)~Setting~A and (b)~Setting~B both exhibit monotone convergence under relaxation $\eta=0.5$ with $M=16$ paths.}
\label{fig:rect-convergence}
\end{figure}

\begin{figure}[pos=htbp]
\centering
\subcaptionbox{Setting~A: predicted (left), exact (center), error (right) at $t=1.0$}{
\includegraphics[width=0.78\linewidth]{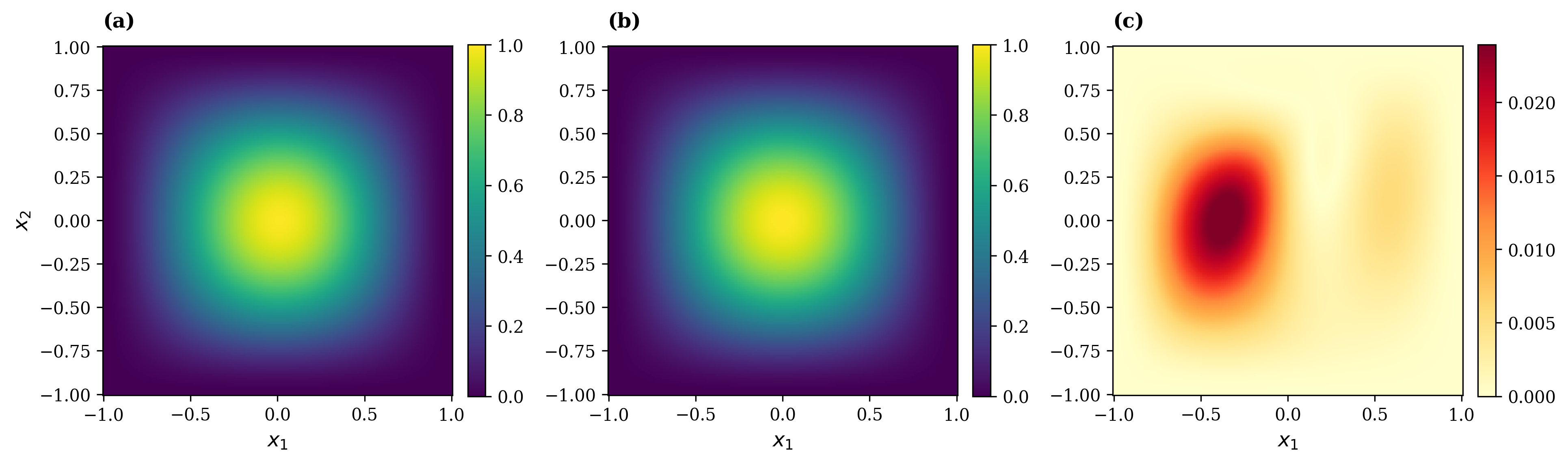}
}
\subcaptionbox{Setting~B: predicted (left), exact (center), error (right) at $t=1.0$}{
\includegraphics[width=0.78\linewidth]{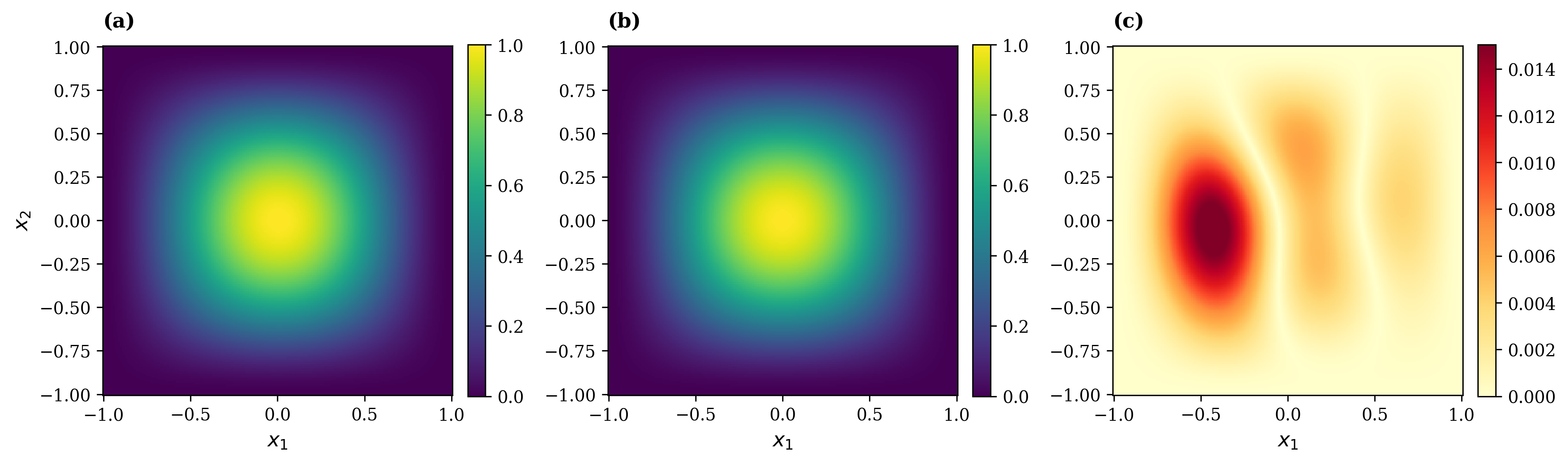}
}
\caption{Solution quality at $t=T=1.0$ for the unit square domain.}
\label{fig:rect-heatmap}
\end{figure}

\subsection{Example 3: Unit disk with double-bump profile}
\label{subsec:doublebump}

We next consider a more heterogeneous benchmark on the unit disk. Unlike the
single-bump profile in Example~\ref{subsec:disk}, the reference solution here
contains two localized components with overlapping supports, which creates a
non-convex spatial structure and a saddle region between the two peaks. This
example is designed to test whether the proposed Deep Picard framework can
resolve spatially localized features while maintaining stable nonlinear
fixed-point iterations.

The exact solution is prescribed as
\begin{equation}\label{eq:exact-sol-doublebump}
u_{\mathrm{ex}}(t,x) = t^\beta\,\psi(x), \qquad
\psi(x) := \varphi_1(x) + \kappa\,\varphi_2(x),
\end{equation}
where
\begin{equation}\label{eq:bump-def}
\varphi_j(x)
=
\Bigl(1 - \tfrac{|x-c_j|^2}{r^2}\Bigr)_+^{1+\alpha/2},
\qquad j=1,2 .
\end{equation}
The centers and relative amplitude are chosen so that the two bumps partially
overlap inside the disk. The exponent $1+\alpha/2$ gives a smoother compactly
supported profile than the boundary factor used in Example~\ref{subsec:disk},
which keeps the corresponding fractional Laplacian bounded near the support
interfaces. Since no closed-form expression is available for
$(-\Delta)^{\alpha/2}\psi$, this term is precomputed using the same
Fourier-multiplier procedure as in the square-domain experiment.

We test the quadratic reaction setting, with the forcing term chosen by
substituting $u_{\mathrm{ex}}$ into the governing equation. Compared with the
single-bump disk problem, this benchmark is substantially more demanding: the
Monte Carlo labels must capture not only the global decay induced by the
fractional dynamics, but also the localized interaction between two separated
spatial features.

The results are shown in Figure~\ref{fig:double-bump}. The learned solution
captures both peaks and the intermediate saddle region, while the largest
errors are concentrated near the support interfaces of the two bumps, where
the spatial profile changes most rapidly. The Monte Carlo convergence curve
follows the expected square-root decay over the stable range, indicating that
sampling variance remains the dominant source of error once the operational
time discretization is sufficiently refined. Compared with the single-bump
case, this example requires a larger sampling budget to obtain stable Picard
updates, which suggests that spatial heterogeneity, rather than the polynomial
degree of the nonlinearity alone, is a main driver of the computational cost.

\begin{figure}[pos=htbp]
\centering
\includegraphics[width=0.88\linewidth]{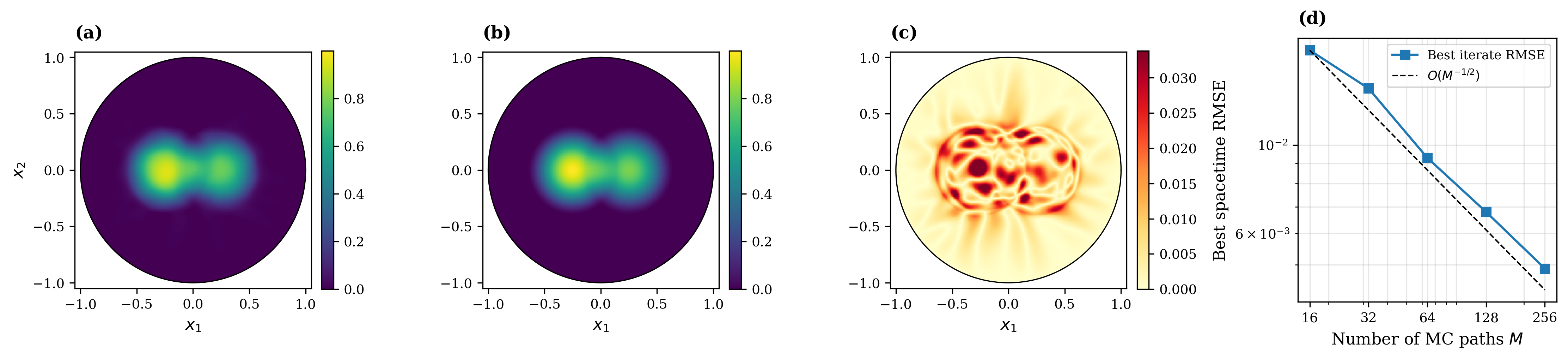}
\caption{Double-bump experiment on the unit disk, showing the learned solution, exact solution, pointwise error, and Monte Carlo path convergence.}
\label{fig:double-bump}
\end{figure}

\subsection{High-dimensional examples}
\label{subsec:highdim}

We finally examine the scalability of the proposed method on unit-ball
problems in dimensions $d=20,50,100$. The goal of this experiment is not only
to test whether the neural approximation can handle high-dimensional inputs,
but also to verify that the trajectory-based stochastic formulation avoids
the tensor-product grids and dense nonlocal discretizations that would be
prohibitive for classical fractional PDE solvers.

The benchmark uses the same radial profile as in Example~\ref{subsec:disk},
\[
u_{\mathrm{ex}}(t,x)=t^\beta(1-|x|^2)_+^{\alpha/2},
\]
for which the fractional Laplacian admits a closed-form expression on the
unit ball. Both reaction settings are tested. As the spatial dimension grows,
the network width and the Monte Carlo sampling budget are increased
accordingly, while the same relaxed Picard strategy is used throughout.

Table~\ref{tab:highdim-results} reports the resulting errors. The method
remains stable across all tested dimensions and both nonlinearities. In
particular, the spacetime RMSE does not deteriorate as the dimension
increases. This behavior is consistent with concentration of measure in
high-dimensional balls: most uniformly sampled points lie near the boundary,
where the radial profile is small, thereby reducing the effective variance of
the regression labels. The two nonlinear settings exhibit nearly identical
accuracy, indicating that in this regime the dominant error contribution comes
from stochastic label generation and regression, rather than from the specific
form of the reaction term.

Figure~\ref{fig:highdim-convergence} further shows that the relaxed Picard
iteration stabilizes rapidly in all dimensions. The two-dimensional
cross-sections in Figure~\ref{fig:highdim-slices} confirm that the radial
solution profile is recovered accurately, with the residual error mainly
localized near the boundary layer of the hard-constrained ansatz. These
results demonstrate that the proposed method retains its mesh-free and
dimension-robust character in high-dimensional fractional diffusion problems.

\begin{table}[pos=htbp]
\centering
\caption{Error metrics for the high-dimensional experiments. The slice RMSE is computed on a two-dimensional cross-section with the remaining coordinates set to zero.}
\label{tab:highdim-results}
\begin{tabular}{clcc}
\toprule
$d$ & Setting & Slice RMSE & Spacetime RMSE \\
\midrule
\multirow{2}{*}{$20$}
 & A ($\lambda u^2$) & $1.72\times 10^{-2}$ & $3.93\times 10^{-3}$ \\
 & B ($\kappa(u{-}u^3)$) & $1.20\times 10^{-2}$ & $3.66\times 10^{-3}$ \\
\midrule
\multirow{2}{*}{$50$}
 & A ($\lambda u^2$) & $1.68\times 10^{-2}$ & $2.74\times 10^{-3}$ \\
 & B ($\kappa(u{-}u^3)$) & $1.62\times 10^{-2}$ & $2.80\times 10^{-3}$ \\
\midrule
\multirow{2}{*}{$100$}
 & A ($\lambda u^2$) & $1.64\times 10^{-2}$ & $1.70\times 10^{-3}$ \\
 & B ($\kappa(u{-}u^3)$) & $1.59\times 10^{-2}$ & $1.69\times 10^{-3}$ \\
\bottomrule
\end{tabular}
\end{table}

\begin{figure}[pos=htbp]
\centering
\includegraphics[width=0.9\linewidth]{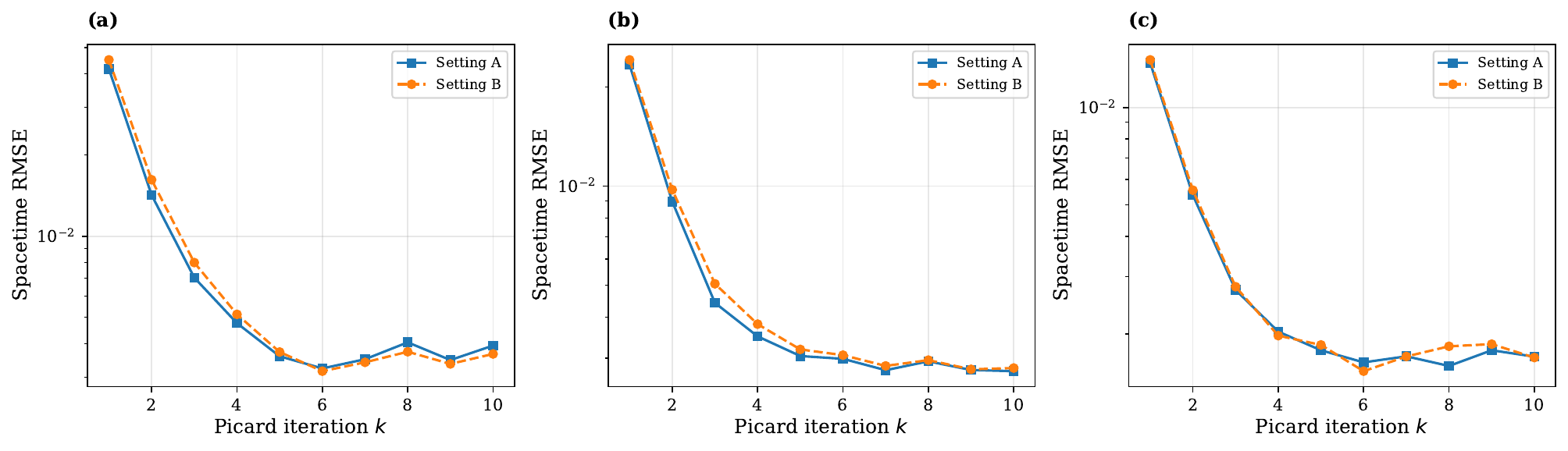}
\caption{Spacetime RMSE versus Picard iteration for the high-dimensional examples.}
\label{fig:highdim-convergence}
\end{figure}

\begin{figure}[pos=htbp]
\centering
\subcaptionbox{$d=20$, Setting~A}{
\includegraphics[width=0.72\linewidth]{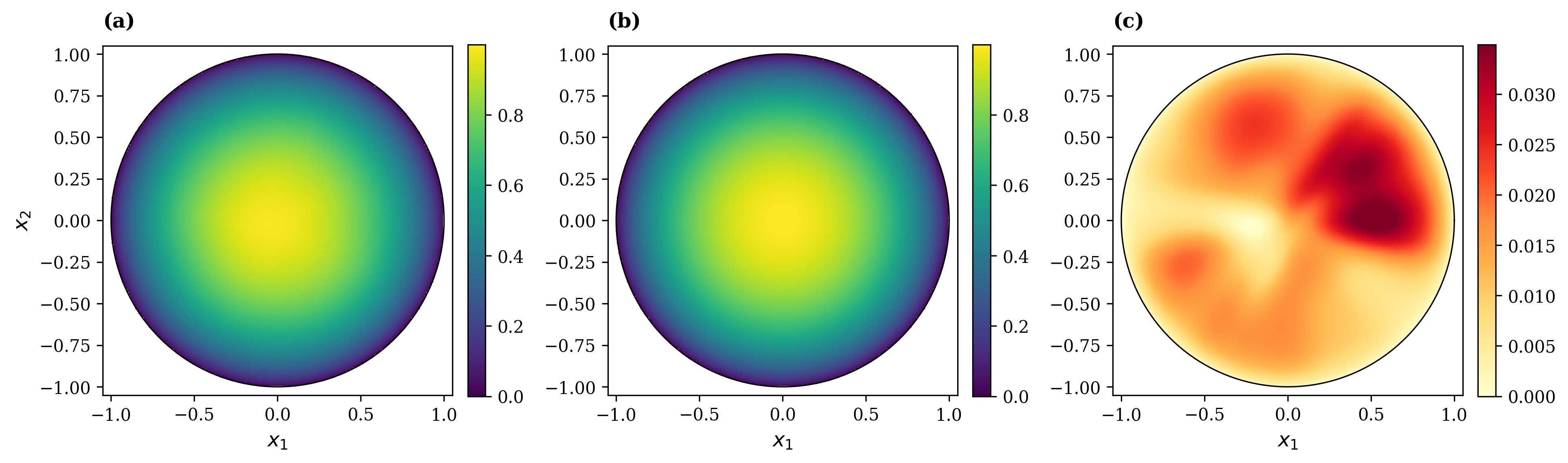}
}
\subcaptionbox{$d=50$, Setting~A}{
\includegraphics[width=0.72\linewidth]{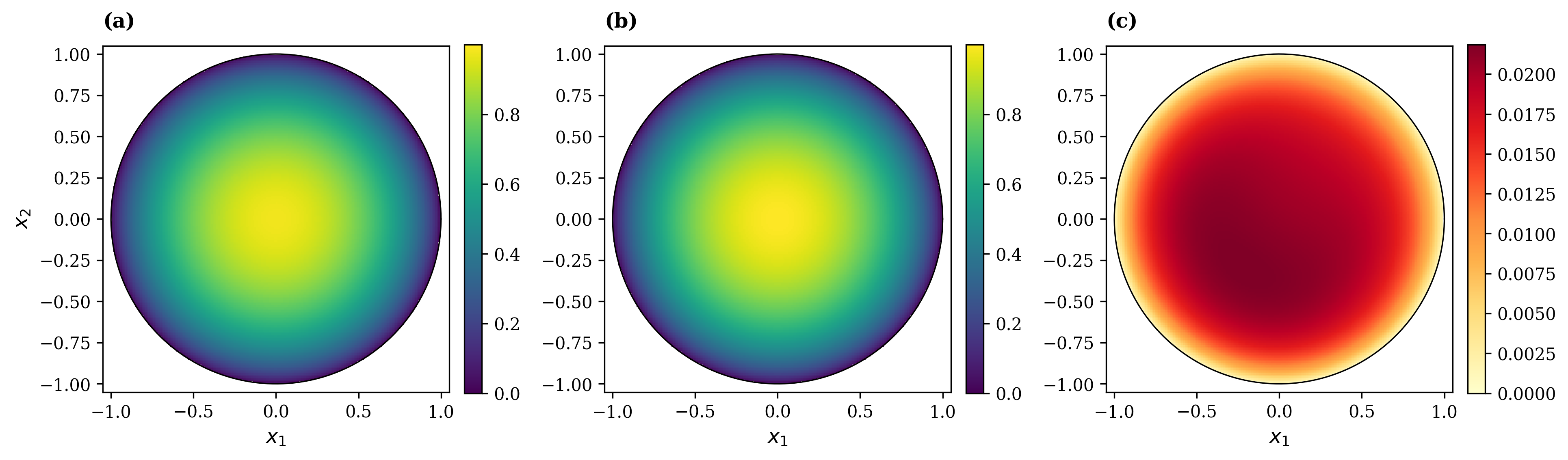}
}
\subcaptionbox{$d=100$, Setting~A}{
\includegraphics[width=0.72\linewidth]{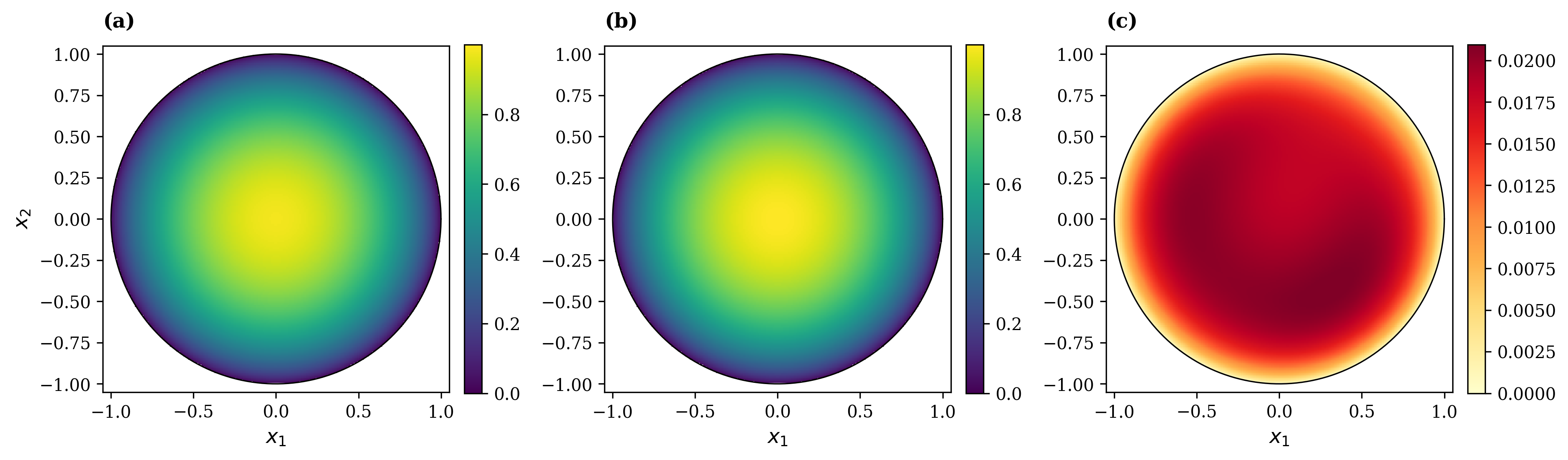}
}
\caption{Two-dimensional cross-sections of the predicted solution, exact solution, and pointwise error at $t=T=1.0$ for Setting~A.}
\label{fig:highdim-slices}
\end{figure}

\section{Conclusion}

In this work, we proposed a Deep Picard iteration framework for solving
high-dimensional nonlinear space-time fractional PDEs. The method is
based on a nonlinear fractional Feynman--Kac formulation, which recasts the
fractional PDE as a stochastic fixed-point problem. Instead of directly
discretizing the Caputo memory term and the nonlocal fractional Laplacian, the
proposed approach approximates successive fixed-point updates by Monte Carlo
simulation of the underlying fractional dynamics and realizes these updates by
supervised neural-network regression. In this way, the nonlinear fractional PDE
is transformed into a sequence of learning problems driven by stochastic
representations.

A key feature of the proposed framework is that it is naturally suited to
nonlinear problems. The nonlinearity enters the stochastic representation
through the fixed-point operator and is handled by successive Picard updates,
rather than by differentiating a global residual involving fractional operators.
This structure separates the simulation of fractional dynamics from the
regression of nonlinear solution updates, leading to a mesh-free and
trajectory-based numerical scheme. It also suggests a possible route toward
more general nonlinear models, including fully nonlinear space-time
fractional equations, provided that suitable stochastic or fixed-point
representations can be constructed.

The numerical experiments demonstrate the effectiveness of the method on a
range of benchmark problems. On two-dimensional disk and square domains, the
method accurately recovers reference solutions under both quadratic and cubic
reaction nonlinearities. The observed convergence with respect to the
operational time step and the number of Monte Carlo trajectories is consistent
with the expected behavior of stochastic discretization and Monte Carlo
averaging. The relaxation strategy further improves the stability of the
Picard iteration, especially when the nonlinear term amplifies sampling and
regression errors. In the double-bump example, the method remains stable for a
more spatially heterogeneous solution profile, while the high-dimensional
experiments up to $d=100$ show that the trajectory-based formulation avoids
tensor-product spatial grids and retains its scalability in high dimensions.

Several directions remain for future work. First, it would be valuable to extend
the framework to more complicated geometries and more general boundary
conditions, where the construction of suitable boundary treatments may require
additional analytical and numerical ingredients. Second, more advanced
fixed-point iteration schemes should be investigated. Possible directions
include adaptive relaxation, Anderson acceleration, Newton--Picard variants, or
variance-aware update rules that account for the stochastic error in the Monte
Carlo labels. Third, the computational cost of label generation grows with the
number of Monte Carlo trajectories and Picard iterations. Multilevel Monte
Carlo, control variates, importance sampling, and other variance-reduction
strategies may substantially improve efficiency while preserving accuracy.
Finally, a rigorous discretization-error analysis remains an important open
problem. Such an analysis should quantify the combined effects of subordinator
discretization, walk-on-spheres approximation, numerical quadrature, Monte
Carlo sampling, neural regression, and fixed-point iteration error. These
questions form a natural basis for further theoretical and algorithmic
development of learning-based solvers for nonlinear fractional PDEs.

\section*{Acknowledgements}
This work was supported by the National Key R\&D Program of China (Grant No.\ 2021YFA0719200) and the National Natural Science Foundation of China (Grant No.\ 12071244).

\bibliographystyle{cas-model2-names}
\bibliography{sample}
\end{document}